\documentclass[final,12pt]{article}
\usepackage[utf8]{inputenc}
\usepackage[T1]{fontenc}

\IfFileExists{definice}{%%%%%%%%%%%%%%%%%%%%%%%%%%%%%%%%%%%%%%%%%%%%%%%%%%%%%%%%%%%%%%%%%%%%%%%%%%%%%%%%%%%%%%%
\usepackage{ifthen}

\ifdefined \bDoNotIncludePackages \else
  \newcommand{\bDoNotIncludePackages}{0}
\fi
\ifdefined \bSkipDocumentSetting \else
  \newcommand{\bSkipDocumentSetting}{0}
\fi
\ifdefined \bDoNotDefineTheorems \else
  \newcommand{\bDoNotDefineTheorems}{0}
\fi

\ifthenelse{\equal{\bDoNotIncludePackages}{1}}{}
{
	\usepackage{graphicx}
	\usepackage{amsmath,amsthm,amssymb,mathtools}
	\usepackage{enumerate}
	\usepackage{cleveref} % clever references \Cref{label} or \cref{label}
	\usepackage{subfig}
	\usepackage{color}
}

\ifthenelse{\equal{\bSkipDocumentSetting}{1}}{}{
\addtolength{\voffset}{-1cm} %shifts everything down
\addtolength{\hoffset}{-1.5cm} %shifts everything to the left
\setlength{\textheight}{22cm} \setlength{\textwidth}{16cm}
}

\def\Z{\mathbb Z}

\def\N{\mathbb N}
\def\A{\mathcal A}
\def\B{\mathcal B}
\def\C{\mathcal C}

\def\P{\mathcal P}

\def\P{\mathcal P}
\def\PT{{\mathcal P}_{\Theta}}

\def\L{\mathcal L}
\def\Lu{{\mathcal L}(\uu)}
\def\uu{\mathbf u}
\def\vv{\mathbf v}
\def\tt{\mathbf t}

\def\Pal{{\rm Pal}}
\def \id {{\rm Id}}
\def\PalT{{\rm Pal}_{\Theta}}
\def\PalG{{\rm Pal}_{G}}

\def \Rk#1 {$\mathcal{R}_{#1}$}
\def \Rext {{\rm Rext}}
\def \Lext {{\rm Lext}}
\def \Bext {{\rm Bext}}
\def \Pext {{\rm Pext}}
\def \PextT {\Pext_{\Theta}}
\def \b {{\rm b}}

\def \FC#1 {
\mathcal{C}
\ifthenelse{\equal{#1}{}}{}{(#1)}
}

\def \PC#1 {
\mathcal{P}_{\Theta}
\ifthenelse{\equal{#1}{}}{}{(#1)}
}

\def \PCn#1 {
\mathcal{P}
\ifthenelse{\equal{#1}{}}{}{(#1)}
}

\def \Tr {R}

\def \gT {\gamma_{\Theta}}

\def \Gg #1#2{#1\text{-}\textrm{tls}(#2)}
\def \Gc #1#2{#1\text{-}\textrm{crw}(#2)}

\def \Gl #1#2{#1\text{-}\textrm{lps}(#2)}

\ifthenelse{\equal{\bDoNotDefineTheorems}{1}}{}
{
\newtheorem{thm}{Theorem}

\newtheorem{corollary}[thm]{Corollary}
\newtheorem{lem}[thm]{Lemma}

\newtheorem{prop}[thm]{Proposition}

\theoremstyle{definition}

\newtheorem{definition}[thm]{Definition}
\newtheorem{example}[thm]{Example}

\crefname{thm}{theorem}{theorems}
\crefname{thrm}{theorem}{theorems}
\crefname{coro}{corollary}{corollaries}
\crefname{example}{example}{examples}
\crefname{lem}{lemma}{lemmas}
\crefname{lmm}{lemma}{lemmas}
\crefname{claim}{claim}{claims}
\crefname{obs}{observation}{observations}
\crefname{proposition}{proposition}{propositions}
\crefname{prop}{proposition}{propositions}
\crefname{defi}{definition}{definitions}

\crefname{theorem}{theorem}{theorems}
\crefname{corollary}{corollary}{corollaries}
\crefname{example}{example}{examples}
\crefname{lemma}{lemma}{lemmas}
\crefname{proposition}{proposition}{propositions}
\crefname{definition}{definition}{definitions}

\theoremstyle{remark}
\newtheorem{remark}[thm]{Remark}

\crefname{example}{example}{examples}
}

\definecolor{sediva}{rgb}{0.85,0.85,0.85}
}{}
\usepackage{subfig}

\usepackage{tikz}
\usetikzlibrary{arrows,shapes}

\begin{document}

%\begin{frontmatter}

\title{Palindromic richness for languages \\ invariant under more symmetries}

%\author{Edita Pelantov\'a\thanks{Czech Technical University in Prague} \and \v St\v ep\'an Starosta\footnotemark[1] \thanks{Czech Technical University in Prague A}}
\author{Edita Pelantov\'a \\ Czech Technical University in Prague\\ Czech Republic \and \v St\v ep\'an Starosta \\ Czech Technical University in Prague \\ Czech Republic}
%\address{Czech Technical University in Prague, Trojanova 13, 120~00 Praha~2, Czech Republic}

\date{}
\maketitle

\begin{abstract}
For a  given finite group $G$ consisting of morphisms and antimorphisms of a free monoid  $\mathcal{A}^*$,
we study  infinite words with language closed under the group $G$.
We focus on the notion of $G$-richness which describes words rich in generalized palindromic factors, i.e., in factors $w$ satisfying $\Theta(w) = w$ for some antimorphism $\Theta \in G$.
We give several equivalent descriptions which are generalizations of know characterizations of rich words (in the terms of classical palindromes) and show two examples of $G$-rich words.
%We give several equivalent descriptions of rich words and show two examples of $G$-rich words.
\end{abstract}

%\begin{keyword}
%generalized palindromes \sep palindromic richness \sep graph of symmetries \sep Coxeter groups


%\MSC 68R15

%\end{keyword}

%\end{frontmatter}

\section{Introduction}

In \cite{DrJuPi}, Droubay et al. showed  that the number of
different palindromes occurring in a finite word $w$ cannot
exceed the bound $|w|+1$. If this  bound is met, the word $w$ is
called rich or rich in palindromes or full \cite{DrJuPi, BrHaNiRe}. An
infinite word $\uu$ is said to be rich if all of its factors are
rich. The list of the most prominent rich words contains
Arnoux-Rauzy words, see \cite{DrJuPi},  and words coding interval exchange with
symmetric permutation of intervals, see \cite{BaMaPe}. Note that Sturmian words
belong to both mentioned classes of words.

During the past two decades, the notion of \textit{palindromic richness}
showed to be fruitful. Application of \textit{palindromes} in physics of quasicrystals
(see for instance \cite{HoKnSi, DaLe_I}) and in genetics (see for instance \cite{KaMa}) served as
stimulating factor for research  in this area as well.
Restivo and Rosone \cite{ReRo09} showed a narrow connection of rich
periodic words with extremal cases of Burrows-Wheeler transform
which is used in  compression algorithms.
In \cite{ReRo}, the authors further refined the result.

Generalizations of rich words appeared soon. Instead of
classical palindromes defined as words invariant under the reversal
mapping one can consider \textit{\mbox{$\Theta$-palindromes}}, i.e., words
invariant under an involutive antimorphism $\Theta$.
For first appearance of the notion see \cite{KaMa2008},
where it appeared in the context of DNA, or \cite{LuLu},
where the name pseudopalindrome is also used.
 Words saturated by $\Theta$-palindromes up to the
highest possible level are called \textit{\mbox{$\Theta$-rich}}. Another kind
of generalization of rich words relaxes the requirement on the number of
palindromes occurring in any factor $w$. We say that infinite word
$\uu$ is \textit{almost rich} if there exists a constant $D$ such that any
factor $w$ of $\uu$ contains at least $|w|+1-D$ different
palindromes. The minimal constant $D$ with this property is referred
to as \textit{palindromic defect} and was introduced in \cite{BrHaNiRe}. Both
mentioned generalizations can be combined into the notion of \textit{almost
$\Theta$-rich} words. It follows directly from the definition that
almost $\Theta$-rich words contain infinitely many
$\Theta$-palindromes.

 As shown in \cite{PeSta1},
besides peculiar periodic words, no infinite uniformly recurrent words can be
simultaneously almost $\Theta_1$-rich and  almost $\Theta_2$-rich
for two distinct involutive antimorphisms $\Theta_1$ and
$\Theta_2$. The famous Thue-Morse word contains infinitely many
classical palindromes and $E$-palindromes, where $E$ is the antimorphism generated by interchange of symbols $0$ and $1$.
Nevertheless, the Thue-Morse word has no chance to be  almost rich
or almost $E$-rich.
 Therefore, the authors suggested in \cite{PeSta1} a further
generalization under the name of \textit{$G$-richness} and \textit{almost $G$-richness}. The new definition
of richness respects more antimorphisms of finite order under which
the language of an infinite word is invariant, the letter $G$
stands for the group generated by these antimorphisms. The
definition is based on the notion of \textit{graph of symmetries}, which is
assigned to an infinite word.   Adopting the new definition, the
second author showed in \cite{Sta2011} that all
generalized Thue-Morse words $\tt_{b,m}$
 are $I_2(m)$-rich, where $I_2(m)$ is a group isomorphic to the dihedral
group having $2m$ elements.

It turned out that words rich in the classical sense can be
characterized by using many other notions of combinatorics on words:
return word, bilateral order, longest palindromic suffix, factor and palindromic complexity, and super reduced Rauzy graphs. These
characterizations can be found in \cite{GlJuWiZa,BuLuGlZa,BaPeSta2}.
Analogous results for $\Theta$-rich words can be found in
\cite{Sta2010}. The aim of this article is to find $G$-analogies of
these characterizations.
They are stated as \Cref{CharakteristikaReturn,CharakteristikaLongestSuffix,CharakteristikaDG} and \Cref{G_rich_Pext_rovnosti}.
A consequence of these
characterizations is the fact that existence of a $G$-rich word
forces the group to be generated by involutive antimorphisms only.
This class of groups contains dihedral groups and, more generally, finite Coxeter groups.
%, for definition see for instance \cite{Hu}.
The question whether there exists a $G$-rich word for any finite group $G$ such that it is generated by involutive antimorphisms, or it is at least a finite Coxeter group, remains open.
At the end of the article we present two examples of $G$-rich words.
The two examples are defined over alphabet of distinct sizes, nevertheless their groups of symmetries are mutually isomorphic, but not to a dihedral group.

The list of known examples of $G$-rich and almost $G$-rich words is very short.
One aim of this article is to trigger a search for such words.
%The generalized palindromic closure introduced in the last chapter of \cite{LuLu} seems to be a promising tool for construction of $G$-rich and almost $G$-rich words.

%======================================================

\section{Preliminaries}
An \textit{alphabet} $\A$ is a finite set.
Elements of $\A$ are usually called \textit{letters}.
A \textit{finite word} $w$ over $\A$ is a finite string $w = w_1w_2\cdots w_n$
of letters $w_i \in \A$.
 Its length, denoted by $|w|$, is $n$.
The set of all finite words over $\A$ equipped
with the operation of concatenation is the free monoid
$\mathcal{A}^*$.
Its neutral element is the \textit{empty word }$\varepsilon$.
A word $v \in \mathcal{A}^*$ is a \textit{factor} of a word $w \in \mathcal{A}^*$ if there exist words $s,t\in \mathcal{A}^*$
such that $w=svt$. If $s=\varepsilon$, then $v$ is a \textit{prefix} of
$w$, if  $t=\varepsilon$, then $v$ is a \textit{suffix} of $w$.

\subsection{Antimorphisms and their fixed points}
A  mapping $\varphi$ on $\mathcal{A}^*$ is called
\begin{itemize}
\item a \textit{morphism} if $ \varphi(vw)=  \varphi(v) \varphi(w)$ for any
$v,w
 \in \mathcal{A}^*$;
\item an \textit{antimorphism} if $ \varphi(vw)=  \varphi(w) \varphi(v)$
for any $v,w \in \mathcal{A}^*$.
\end{itemize}
We denote the set of all morphisms and antimorphisms on
$\mathcal{A}^*$ by $AM({\mathcal{A}^*})$. Together with
composition, it forms a monoid with the identity mapping $\id$
as the unit element.
 The set of all morphisms, denoted by $M({\mathcal{A}^*})$, is a
submonoid of $AM({\mathcal{A}^*})$. The \textit{reversal} mapping $\Tr$
defined by $$\Tr(w_1w_2\cdots w_n) = w_nw_{n-1}\cdots w_2w_1 \quad \text{for all } w = w_1\cdots w_n \in \A^*$$ is
an involutive antimorphism, i.e., $\Tr^2 = \id$. It is obvious
that any antimorphism  is a composition of $\Tr$ and a morphism.
Thus
$$AM({\mathcal{A}^*})= M({\mathcal{A}^*}) \cup
\Tr \big( M({\mathcal{A}^*}) \big).$$

A fixed point of a given antimorphism $\Theta$  is called
\textit{\mbox{$\Theta$-palindrome}}, i.e., a word $w$ is a  \mbox{$\Theta$-palindrome} if
$w=\Theta(w)$. If $\Theta$ is the reversal mapping $\Tr$, we say
palindrome or classical palindrome  instead of $\Tr$-palindrome.
One can see that if $\Theta$ has a fixed point containing all the
letters of $\A$, then $\Theta$ is an involution, and thus a
composition of $\Tr$ and an involutive permutation of letters.

\subsection{Factor and palindromic complexities }

An \textit{infinite word} $\uu$ over an alphabet $\mathcal{A}$ is a
sequence $(u_n)_{n\in \mathbb{N}}\in \mathcal{A}^\mathbb{N}$.
We always implicitly suppose that $\A$ is the smallest
possible  alphabet for $\uu$, i.e., any letter of $\A$ occurs at
least once in $\uu$. Action of any morphism $\varphi\in   M({\mathcal{A}^*})$ can be naturally extended to infinite words
 by the  prescription
$$\varphi(\uu) = \varphi(u_0)\varphi(u_1) \varphi(u_2) \ldots \quad \text{ for all } \uu = (u_n)_{n\in \N} \in \A^{\N}. $$
 A finite word $w$ is a factor of $\uu$ if there exists an index
$i\in \mathbb{N}$, called \textit{occurrence} of $w$, such that $w=
u_iu_{i+1}\cdots u_{i+|w|-1}$. The set of all factors of $\uu$ of
length $n$ is denoted  $\mathcal{L}_n(\uu)$. The \textit{language} of an
infinite  word  $\uu$ is the set of all of its factors
$\mathcal{L}(\uu)= \cup_{n\in\mathbb{N}}\mathcal{L}_n(\uu)$. An
infinite word $\uu$ is \textit{recurrent} if any of  its factors has infinitely many occurrences in $\uu$. A factor $v\in\mathcal{L}(\uu)$ is a
\textit{complete  return word} of a factor $w$ if  $w$ occurs in $v$
exactly twice, as  a suffix and a prefix of $v$. A complete return
word $v$ of $w$ can be written as $v=qw$ for some factor
$q\neq\varepsilon$, which is usually called a \textit{return word} of $w$.
If any factor $w$ of $\uu$ has only finitely many return words,
then $\uu$ is said to be \textit{uniformly recurrent}.

The \textit{factor complexity} of $\uu$ is the mapping $\mathcal{C}:
\mathbb{N}\mapsto \mathbb{N}$ defined by the prescription
$$\mathcal{C}(n):=\# \mathcal{L}_n(\uu).$$
To evaluate the factor complexity of $\uu$, one may watch possible
prolongations of factors. A letter $a\in  \mathcal{A}$ is a \textit{left
extension} of a factor $w$ in $\uu$  if $aw$ belongs to
$\mathcal{L}(\uu)$. The set of all left extensions of $w$ is
denoted $\Lext (w)$. A factor $w \in \mathcal{L}(\uu)$ is called
\textit{left special} if  $\#\Lext (w)\geq 2$. Analogously, we define
\textit{right extension}, the set $\Rext (w)$, and \textit{right special}.  If
$w$ is right and left special, we call it \textit{bispecial}. The
first difference of the factor complexity of a recurrent word
$\uu$ satisfies
\begin{equation*}\label{DeltaC} \Delta \mathcal{C}(n) = \mathcal{C}(n+1) -
\mathcal{C}(n) = \sum_{w\in \mathcal{L}_n(\uu)}  \bigl(\# \Lext
(w) -1\bigr) =\sum_{w\in \mathcal{L}_n(\uu)}  \bigl(\# \Rext (w)
-1\bigr).
\end{equation*}
The second difference of factor complexity can be expressed using
the \textit{bilateral order} of a factor.
Let $w$ be a factor of $\uu$.
Its bilateral order is the quantity $\b(w): = 
\# \Bext(w) - \# \Lext (w) -\# \Rext (w)+1$,
where $\Bext(w) = \{ awb \ | \ awb \in \L(\uu), a,b \in \A\}$.
In \cite{Ca}, the formula
\begin{equation*}\label{DeltaNaDruhuC} \Delta^2 \mathcal{C}(n) = \Delta \mathcal{C}(n+1) -
\Delta \mathcal{C}(n) = \sum_{w\in \mathcal{L}_n(\uu)} \b(w)
\end{equation*}
is deduced.

The \textit{$\Theta$-palindromic complexity} of $\uu$ is the mapping
$\PT(n): \N \mapsto \N$ defined by
$$
\PT(n) := \# \{ w \in \L_n(\uu) \mid w = \Theta(w) \}.
$$
Similarly to factor  complexity, evaluation of palindromic
complexity is based on counting  possible extensions of
palindromes. For a $\Theta$-palindrome $w$, we denote by
$\Pext_\Theta(w)$ the set of all letters $  a \in \mathcal{A}$
such that $aw\Theta(a) \in \L(\uu)$. It is easy to see that

%\begin{equation}\label{DeltaP} \PT(n+2) - \PT(n) = \sum\limits_{\substack{w \in \L_n(\uu)\\ w=
%\Theta(w)}}\left(\#\Pext_{\Theta}(w)-1\right).
%\end{equation}

\begin{equation}\label{DeltaP} \PT(n+2) = \sum\limits_{\substack{w \in \L_n(\uu)\\ w=
\Theta(w)}} \#\Pext_{\Theta}(w).
\end{equation}

%======================================================

\subsection{Words with language closed under a group $G$ and $G$-richness}

In the rest of the article, the symbol $G$ stands exclusively for a subset of $AM({\mathcal{A}^*})$  satisfying the two following requirements:
\begin{enumerate}[i)]
\item \label{reqG1} $G$ is a finite group;
\item \label{reqG2} $G$ contains at least one antimorphism.
\end{enumerate}

The first requirement on $G$ implies the following for an element $\nu$ of $G$.
The element $\nu$ is non-erasing, i.e., $\nu(a) \neq \varepsilon$ for all $a \in \A$ (otherwise $\nu$ has no inverse in $G$).
Moreover, $\nu(a)$ is a letter for all $a \in \A$ (otherwise $\nu^n \neq \id$ for all $n \geq 1$).
We can conclude that $\nu$ restricted to $\A$ is a permutation of letters.

The second requirement on $G$ stems from the fact that our results are based on generalized palindromes and one gets only trivial or no results when dealing with groups consisting of morphisms only.
Since especially involutive antimorphisms are important in the study of generalized palindromes, 
by $G^{(2)}$ we denote the set of all involutive antimorphisms belonging to $G$. 

\begin{example}
Set $\A = \{0, 1\}$. Denote by $E$ the antimorphism determined by $0 \mapsto 1$ and $1 \mapsto 0$.
The only finite subgroups of $AM({\mathcal{A}^*})$ containing at least one antimorphism are
$$
G_1 = \{\id, \Tr\}, \ G_2 = \{\id, E\}, \text{ and } G_3 = \{\id,E,\Tr,E\Tr\}.
$$
In this case, we have
$$
G_1^{(2)} = \{\Tr\}, \ G_2^{(2)} = \{E\}, \text{ and } G_3^{(2)} = \{E,\Tr\}.
$$
\end{example}

The previous example shows that binary alphabet allows only abelian groups to satisfy \ref{reqG1}) and \ref{reqG2}).
On multiliteral alphabet, $G$ need not be abelian.

\begin{example}
Let $m$ be an integer such that $m \geq 3$.
Let $\pi \in S_m$ be a permutation of $\A = \{0,1,\ldots,m-1\}$.
Denote by $\mu_{\pi}$ the morphism on $\A^*$ induced by $\pi$, i.e., the restriction of $\mu_{\pi}$ to $\A$ is the permutation $\pi$ of $\A$.
Then the set
$$
G = \{ \mu_{\pi} \mid \pi \in S_m \} \cup \{ R\mu_{\pi} \mid \pi \in S_m \}
$$
is a group satisfying \ref{reqG1}) and \ref{reqG2}).
Clearly, $G$ is not abelian.
\end{example}

\begin{example} \label{ex:I2_3}
Let $\A = \{0, 1, 2\}$.
For all $k \in \A$ define the antimorphism $\Psi_k$ by $\Psi_k(\ell) = (k - \ell) \bmod{3}$ for all $\ell \in \A$.
By $\mu$ denote the morphism determined by $\mu(\ell) = (\ell - 1) \bmod{3}$ for all $\ell  \in \A$.
The set
$$
G = \{ \id, \mu, \mu^{-1}, \Psi_0, \Psi_1, \Psi_2 \}
$$
forms a non-abelian group satisfying \ref{reqG1}) and \ref{reqG2}) and not containing $R$.
\end{example}

Let us stress some aspects of such a group $G$ satisfying requirements \ref{reqG1}) and \ref{reqG2}):
\begin{enumerate}
\item every element of $G$ is either a morphism or an antimorphism determined by a permutation of letters of $\A$;
\item $G$ may contain elements of order greater than $2$;
\item $G$ need not be abelian;
\item \label{G-aspects-gens} the set of antimorphisms of $G$ generates the group $G$;
\item \label{G-aspects-no} the number of morphism in $G$ equals the number of antimorphisms in $G$;
\end{enumerate}

To prove the last two items it suffices to fix an antimorphism $\Theta \in G$.
Item \ref{G-aspects-gens} follows from the fact that given a morphism $\mu \in G$ one can write
 $\mu \Theta = \Theta'$ where $\Theta'$ is an antimorphism of $G$.
Thus, $\mu = \Theta' \Theta^{-1}$.
To show the last property, it suffices to show that the mapping from the set of morphisms of $G$ to the set of antimorphisms of $G$ defined by $\mu \mapsto \mu \Theta$ for all morphism $\mu \in G$ is a bijection.

We say that finite words $w, v\in
\A^*$ are \textit{$G$-equivalent} if there exists $\mu \in G$ such that
$w=\mu(v)$. The class of equivalence containing a word $w$ is
denoted
 $$ [w] := \left \{ \mu(w) \mid \mu
\in G \right \}.
$$
As already mentioned, since the group $G$ is
finite, any $\mu \in  G$ preserves  length  of words and thus
equivalent words have the same length.

 We say
that language $\L(\uu)$ of an infinite word $\uu \in
\A^\mathbb{N}$ is \textit{closed under $G$} if for any factor $w \in
\L(\uu)$ and any $\mu\in G$, the image   $\mu(w)$ belongs to
$\L(\uu)$ as well.  Since $G$ contains at least one antimorphism,
closedness of $\Lu$ under $G$ implies that $\uu$ is recurrent.

A useful tool to study language in combinatorics on words is \emph{Rauzy graph}.
Given a language $\L$, the Rauzy graph of order $n$ of the language $\L$ is a subgraph of $n$-dimensional De Bruijn graph such that the set of vertices equals $\L_{n} = \L \cap \A^n$ and the set of edges equals $\L_{n+1} = \L \cap \A^{n+1}$.
In other words, there is an oriented edge $e \in \L_{n+1}$ starting in $p \in \L_n$ and ending in $s \in \L_n$ if $p$ is a prefix of $e$ and $s$ is a suffix of $e$.
For languages closed under reversal, the notion of Rauzy graph has been generalized in \cite{BuLuGlZa} to \emph{super reduced Rauzy graph}.
A super reduced Rauzy graph is undirected and has multiple edges and loops allowed.
It can be produced from a Rauzy graph by a certain ``compression'' which replaces some vertices and takes advantage of the symmetry given by the reversal mapping, see \cite{BuLuGlZa} for more details.
This process is lossless and one can reconstruct the Rauzy graph back.
The following definition (introduced in \cite{PeSta1}) of undirected graph of symmetries generalizes the notion of super reduced Rauzy graph; the two definitions coincide for $G = \{\id, \Tr\}$.
% The following graph assigned to such a word $\uu$ was defined in \cite{PeSta1}.

\begin{definition} Let
 $\uu$ be an infinite word with language closed under
$ G$ and $n\in \mathbb{N}$.

\begin{description}
\item[1)] The \emph{directed graph  of symmetries of the word $\uu$ of order $n$} is
$\overrightarrow{\Gamma}_n(\uu) = (V,\overrightarrow{E})$ with the
set of vertices $$V=\{ [w]\mid w\in \L_n(\uu), w \ \hbox{is left or
right special} \}$$ and  an edge $e\in \overrightarrow{E} \subset \L(\uu)$
starts in the vertex $[w]$ and ends in the vertex $[v]$ if and only if
\begin{itemize}
\item the prefix of $e$ of length $n$ belongs to $[w]$,
 \item the suffix of $e$ of length $n$ belongs to $[v]$,
\item $e$ has exactly two occurrences of special factors of length $n$, i.e., $e$ is at least of length $n+1$ and all its factors of length $n$  except for its prefix and suffix are not special.
\end{itemize}

\item[2)] The \emph{undirected graph of symmetries of the word $\uu$ of order $n$} is
$\Gamma_n(\uu)=(V,E)$ with the same set of vertices as
$\overrightarrow{\Gamma}_n(\uu)$ and two vertices $[w]$ and $[v]$ are connected by an undirected edge $[e] \in E$ if and only if
$$
e \in \overrightarrow{E} \text{ starts in } [w] \text{ and ends in } [v] \text{ or vice versa.}
$$
\end{description}
Multiple edges and loops are allowed in both defined graphs.
\end{definition}

Any vertex $[w]$ of the graph $\overrightarrow{\Gamma}_n(\uu)$
represents a class of equivalent factors of $\L_n(\uu)$. It has at
most  $\#G$ elements; the actual cardinality  of $[w]$ may depend
on $n$ as well.

Since $G$ contains at least one antimorphism, the word $\uu$ is recurrent which implies that both graphs  $\overrightarrow{\Gamma}_n(\uu)$ and  $\Gamma_n(\uu)$ are connected.

We give two famous examples to demonstrate the last definition.

\begin{example}[The Fibonacci word]
The Fibonacci word $\uu_F$ is the fixed point of the morphism determined by
$$
0 \mapsto 01, 1 \mapsto 0.
$$
We have
$$
\uu_F = 0100101001001010010100100101001001010010 \ldots \ .
$$
The language of the Fibonacci word is closed under reversal, i.e., if we set $G_F = \{ \id, \Tr \}$, then the language $\L_{\uu_F}$ is closed under $G_F$.
We have
$$
\L_3(\uu_F) = \{ 101, 010, 100, 001 \} \quad \text{ and } \quad \L_4(\uu_F) = \{ 1001, 1010, 0100, 0010, 0101 \}.
$$
The only special factor of length $3$ is $010$ and it is in fact bispecial.

Figure \ref{fig:fibo_example_directed} depicts the graph $\overrightarrow{\Gamma}_3(\uu_F)$,
while Figure \ref{fig:fibo_example_undirected} shows the graph ${\Gamma}_3(\uu_F)$.

% overright arrow is not working if placed below...
\newsavebox{\helperfibo}
\savebox{\helperfibo}{$\overrightarrow{\Gamma}_3(\uu_F)$}

\begin{figure}[h]
\centering
\parbox{0.45\textwidth}{
\centering
\begin{tikzpicture}[scale=1]

\tikzstyle{every node} = [draw, circle, line width=2pt, outer sep=2pt]

%\node (101) at (1,0) {$[101]$};
\node (010) at (1,3) {$[010]$};

\path[->,line width=2pt,every node/.style={}]
% (101) edge[bend right] node[pos=0.5,right] {$1010$} (010)
 %(010) edge[bend right] node[pos=0.5,left] {$0101$} (101)
 (010) edge[out=60, in=120, loop, looseness=8.5] node[pos=0.5,above] {$010010$} (010)
 (010) edge[out=-60, in=-120, loop, looseness=8.5] node[pos=0.5,below] {$01010$} (010)
;

\end{tikzpicture}
\caption{The graph \usebox{\helperfibo} for the group $G_F$.}
\label{fig:fibo_example_directed}
}%
\qquad
\begin{minipage}{0.45\textwidth}
\centering
\begin{tikzpicture}[scale=1]

\tikzstyle{every node} = [draw, circle, line width=2pt, outer sep=2pt]

%\node (101) at (1,0) {$[101]$};
\node (010) at (1,3) {$[010]$};

\path[-,line width=2pt,every node/.style={}]
 %(101) edge node[pos=0.5,right] {$[1010]$} (010)
 (010) edge[out=60, in=120, loop, looseness=8.5] node[pos=0.5,above] {$[010010]$} (010)
(010) edge[out=-60, in=-120, loop, looseness=8.5] node[pos=0.5,below] {$[01010]$} (010)
 ;

\end{tikzpicture}
\caption{The graph ${\Gamma}_3(\uu_F)$ for the group $G_F$.}
\label{fig:fibo_example_undirected}
\end{minipage}
\end{figure}

\end{example}

\begin{example}[Generalized Thue-Morse words] \label{ex:genTM}

The generalized Thue-Morse word $\tt_{b,m}$ is defined on the alphabet $\{0, \ldots, m-1\}$ for all $b \geq 2$ and $m \geq 1$ as
$$
\tt_{b,m} = \left ( s_b(n) \mod m \right )_{n=0}^{+\infty},
$$
where $s_b(n)$ denotes the sum of digits in the base-$b$ representation of the integer $n$.
See for instance \cite{AlSh} where this class of words is studied.
In \cite{Sta2011}, it is show that the language of $\tt_{b,m}$ is closed under a group isomorphic to the dihedral group of order $2m$, denoted $I_2(m)$.
We exhibit here the graphs of symmetries for two pairs of parameters $b$ and $m$.

Take $b = m = 2$, the word $\tt_{2,2}$ is then the famous Thue-Morse word.
It starts with $0$ and it is a fixed point of the morphism determined by
$$
0 \mapsto 01 \text{ and } 1 \mapsto 10.
$$
We have
$$
\tt_{2,2} = 0110100110010110100101100110100110010110 \ldots
$$

Figure \ref{fig:example-TM-Rauzy} shows the Rauzy graph of order $3$ of the Thue-Morse word.
%Its vertices are the factors of length $3$ and there is an edge from $w$ to $v$ if there exists a factor of length $4$ such that $w$ is its prefix and $v$ its suffix.

The language of the Thue-Morse word is closed under the reversal mapping and the antimorphism exchanging letters, denoted again $E$.
Thus, it is closed under the group
$$
I_2(2) = \{\id, \Tr, E, E \Tr \}.
$$
Figure \ref{fig:tm_example_directed} depicts the graph $\overrightarrow{\Gamma}_3(\tt_{2,2})$,
while Figure \ref{fig:tm_example_undirected} shows the graph ${\Gamma}_3(\tt_{2,2})$.

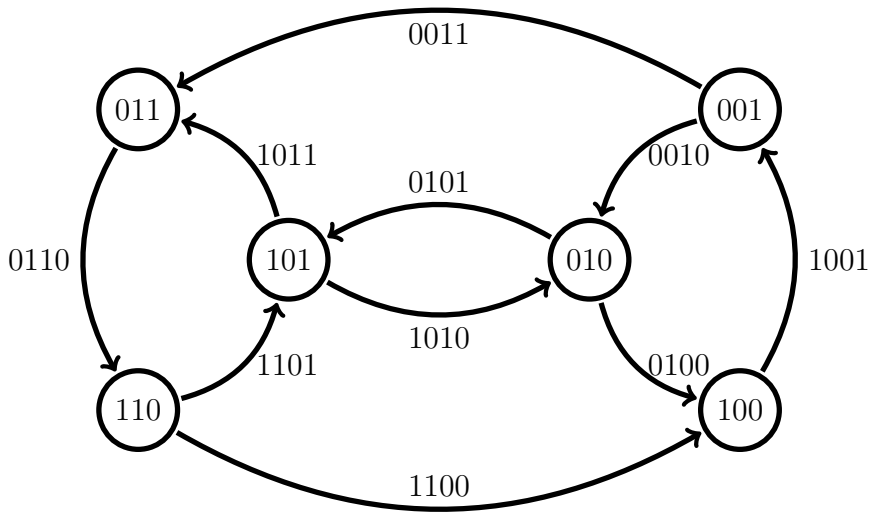
\begin{figure}[h]
\centering
\begin{tikzpicture}[scale=2]

\tikzstyle{every node} = [draw, circle, line width=2pt, outer sep=2pt]

%\only<1-3>{
\node (101) at (1,1) {$101$};
\node (010) at (3,1) {$010$};
\node (100) at (4,0) {$100$};
\node (001) at (4,2) {$001$};
\node (011) at (0,2) {$011$};
\node (110) at (0,0) {$110$};
%}

%\only<-4>{
\path[->,line width=2pt,every node/.style={}]
    (101) edge[bend right] node[pos=0.5,below] {$1010$} (010)
   (010) edge[bend right] node[pos=0.5,above] {$0101$} (101)
(010) edge[bend right] node[pos=0.5,right] {$0100$} (100)
(001) edge[bend right] node[pos=0.5,right] {$0010$} (010)
(101) edge[bend right] node[pos=0.5,right] {$1011$} (011)
(110) edge[bend right] node[pos=0.5,right] {$1101$} (101)

(100) edge[bend right] node[pos=0.5,right] {$1001$} (001)
(001) edge[bend right] node[pos=0.5,below] {$0011$} (011)
(011) edge[bend right] node[pos=0.5,left] {$0110$} (110)
(110) edge[bend right] node[pos=0.5,above] {$1100$} (100)
    %(1) edge[out=-120, in=120, loop, looseness=7.5] node[pos=0.5,left] {$1111$} (1)
    ;
%}

\end{tikzpicture}
\caption{The Rauzy graph of order $3$ of the Thue-Morse word $\tt_{2,2}$.}
\label{fig:example-TM-Rauzy}
\end{figure}

% overright arrow is not working if placed below...
\newsavebox{\helpertm}
\savebox{\helpertm}{$\overrightarrow{\Gamma}_3(\tt_{2,2})$}

\begin{figure}[h]
\centering
\parbox{0.45\textwidth}{
\centering
\begin{tikzpicture}[scale=1]

\tikzstyle{every node} = [draw, circle, line width=2pt, outer sep=2pt]

\node (011) at (1,-0.5) {$[011]$};
\node (101) at (1,3.5) {$[101]$};

\path[->,line width=2pt,every node/.style={}]
%cervene
 (101) edge[bend right=80] node[pos=0.5,left] {$0100$} (011)
 (101) edge[bend right=10] node[pos=0.5,left] {$1011$} (011)
 (011) edge[bend right=10] node[pos=0.5,right] {$0010$} (101)
 (011) edge[bend right=80] node[pos=0.5,right] {$1101$} (101)
%modre
 (011) edge[out=-20, in=-40, loop, looseness=12] node[pos=0.5,below] {$0011$} (011)
 (011) edge[out=-60, in=-80, loop, looseness=12] node[pos=0.5,below] {$1100$} (011)
 (011) edge[out=-120, in=-100, loop, looseness=12] node[pos=0.5,below] {$0110$} (011)
 (011) edge[out=-160, in=-140, loop, looseness=12] node[pos=0.5,below] {$1001$} (011)
%zelene
 (101) edge[out=30, in=60, loop, looseness=14] node[pos=0.5,above] {$0101$} (101)
 (101) edge[out=150, in=120, loop, looseness=14] node[pos=0.5,above] {$1010$} (101)
;

\end{tikzpicture}
\caption{The graph \usebox{\helpertm} for the group $I_2(2)$.}
\label{fig:tm_example_directed}
}%
\qquad
\begin{minipage}{0.45\textwidth}
\centering
\begin{tikzpicture}[scale=1]

\tikzstyle{every node} = [draw, circle, line width=2pt, outer sep=2pt]

\node (011) at (1,0) {$[011]$};
\node (101) at (1,3) {$[101]$};

\path[-,line width=2pt,every node/.style={}]
 (011) edge node[pos=0.5,right] {$[0100]$} (101)
 (101) edge[out=60, in=120, loop, looseness=8.5] node[pos=0.5,above] {$[1010]$} (101)
 (011) edge[out=-30, in=-80, loop, looseness=9.5] node[pos=0.5,below] {$[1100]$} (011)
 (011) edge[out=-100, in=-150, loop, looseness=9.5] node[pos=0.5,below] {$[1001]$} (011)
 ;

\end{tikzpicture}
\caption{The graph ${\Gamma}_3(\tt_{2,2})$ for the group $I_2(2)$.}
\label{fig:tm_example_undirected}
\end{minipage}
\end{figure}

Take $b = m = 3$.
As shown in \cite{AlSh}, the word $\tt_{3,3}$ is a fixed point of the morphism given by
$$
0 \mapsto 012, \ 1 \mapsto 120 \text{ and } 2 \mapsto {201}
$$
and starting with $0$.
Its language is closed under the group $I_2(3)$ which coincides with the group already introduced in \Cref{ex:I2_3}.

We have
$$
\L_3(\tt_{3,3}) = \{001, 202, 220, 200, 201, 011, 010, 012, 020, 122, 112, 101, 120, 212, 121 \},
$$
where the only special factors are $012$, $120$, and $201$.
Figure \ref{fig:tmm_example_undirected} shows the undirected graph of symmetries $\Gamma_3(\tt_{3,3})$.

\begin{figure}[h]
\centering
\begin{tikzpicture}[scale=1]

\tikzstyle{every node} = [draw, circle, line width=2pt, outer sep=2pt]

\node (012) at (1,0) {$[012]$};

\path[-,line width=2pt,every node/.style={}]
 (012) edge[out=60, in=120, loop, looseness=10] node[pos=0.5,above] {$[012120]$} (012)
 (012) edge[out=-60, in=-120, loop, looseness=10] node[pos=0.5,below] {$[0120]$} (012)
;

\end{tikzpicture}
\caption{The graph ${\Gamma}_3(\tt_{3,3})$ for the group $I_2(3)$.}
\label{fig:tmm_example_undirected}
\end{figure}
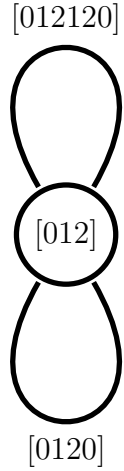

\end{example}

\begin{definition} \label{def:Gdistinguishing}
Let $\uu$ be an infinite word with language
closed under $G$. We say that a number $n \in \mathbb{N}$ is
\emph{$G$-distinguishing} on $\uu$ if  for any $w \in \L_n(\uu)$ we have:
 \begin{equation}\label{anti}
\hbox{$\Theta_1\neq \Theta_2 \quad  \Rightarrow   \quad
\Theta_1(w) \neq \Theta_2 (w)$ \quad  for  any two antimorphisms
$\Theta_1, \Theta_2 \in G$.}
\end{equation}
\end{definition}
If  $n$ is $G$-distinguishing on $\uu$, then the knowledge of a pair
$w$ and $\Theta(w)$   for a single word $w \in \L_n(\uu)$ enables
us to unambigously determine $\Theta \in G$. Let us stress that the requirement
\eqref{anti} gives also for  any two distinct morphisms $\phi_1,
\phi_2 \in G$ that  $\phi_1(w) \neq \phi_2 (w)$. 
(It suffices to consider antimorphisms $\phi_1\Theta$ and $\phi_2\Theta$, where $\Theta \in G$ is an antimorphism,
and use the fact that for all $w \in \Lu$ there is a factor $w' \in \Lu$ such that $\Theta(w') = w$.)
%Thus, $n$ is the length of words on which action of any element of the group $G$ is fully demonstrated.

One can readily see that if $n$ is $G$-distinguishing on $\uu$, then any $m$ greater than $n$ is also $G$-distinguishing on $\uu$.

As proved in \cite{PeSta1}, the connectivity of graphs of symmetries, which follows from the closedness of the language under $G$ containing an antimorphism,
enables to bind the factor and palindromic complexities.

\begin{thm}\label{nerovnostProVice} If
$\uu$ is an infinite  word with language closed  under a group $G$ and
 $N\in \mathbb{N}$ is $G$-distinguishing on $\uu$,
then
\begin{equation}\label{mistoStromovi} \Delta \C (n) + \# G \ \ \geq  \sum_{\Theta \in G^{(2)}}\Bigl(\P_{\Theta}(n) + \P_{\Theta}(n+1)\Bigr) \qquad
\hbox{for any} \ n \geq N.
\end{equation}
\end{thm}

The term  $\Delta \C (n) + \# G$ represents an upper bound on the number of palindromes occurring in $\uu$.
It follows from the proof of the last theorem in \cite{PeSta1} that, for a given integer $n$, the equality in \eqref{mistoStromovi} is reached if and only if the undirected graph of symmetries $\Gamma_n(\uu)$ has a specific tree-like structure.
 Words for which this upper bound is reached are in some sense opulent in palindromes.

Therefore, we adopt this specific structure in the following definition of Property $G$-tls$(N)$, where ``tls'' is an abbreviation of tree-like structure.
The name of the property also keeps track of the group $G$ since a word may satisfy it or not according to the choice of $G$ as we illustrate on the Thue-Morse word just after the definition.
The parameter $N$  plays a similar role as in \Cref{nerovnostProVice} -- it enables us to relate the property of having a tree-like structure to a value measuring the deficit of generalized palindromes.

\begin{definition}
 We say
that  an infinite  word $\uu$ has \emph{Property $G$-tls$(N)$}  if for
each $n\in \mathbb{N}$, $n\geq N$  we have that
\begin{itemize}
\item $\L(\uu)$ is closed  under  $ G$;
  \item if $[e]$ is a loop in $\Gamma_n(\uu)$, then $e$ is a
  $\Theta$-palindrome for some $\Theta\in G$;
  \item the graph obtained from $\Gamma_n(\uu)$ by removing loops is a tree.
  \end{itemize}

\end{definition}

The graph ${\Gamma}_3(\tt_{2,2})$ for the group $G = I_2(2)$ in Figure \ref{fig:tm_example_undirected} has tree-like structure.
As shown in \cite{PeSta1}, all the undirected graphs of symmetries of the Thue-Morse word $\tt_{2,2}$ of order greater than $0$ have the structure and thus the word $\tt_{2,2}$ satisfies Property $\Gg{G}{1}$.
However, for $G = \{\id, \Tr\}$ the undirected graphs of symmetries of the Thue-Morse word does not have tree-like structure. \Cref{fig:TM_no_tls} shows the undirected graph $\Gamma_3(\tt_{2,2})$ for $G = \{\id, \Tr\}$.

\begin{figure}[h]
\centering
\begin{tikzpicture}[scale=1]

\tikzstyle{every node} = [draw, circle, line width=2pt, outer sep=2pt]

\node (011) at (-5,0) {$[011]$};
\node (101) at (-1.5,0) {$[101]$};
\node (010) at (1.5,0) {$[010]$};
\node (001) at (5,0) {$[001]$};

\path[-,line width=2pt,every node/.style={}]
 (011) edge node[pos=0.5,above] {$[1011]$} (101)
 (101) edge node[pos=0.5,above] {$[0101]$} (010)
 (010) edge node[pos=0.5,above] {$[0010]$} (001)
(001) edge[bend left] node[pos=0.5,below] {$[0011]$} (011)
 (011) edge[out=-135, in=135, loop, looseness=5] node[pos=0.5,left] {$[0110]$} (011)
 (001) edge[out=-45, in=45, loop, looseness=5] node[pos=0.5,right] {$[1001]$} (001)
;

\end{tikzpicture}
\caption{The graph $\Gamma_3(\tt_{2,2})$ for the group $\{\id, \Tr\}$.}
\label{fig:TM_no_tls}
\end{figure}
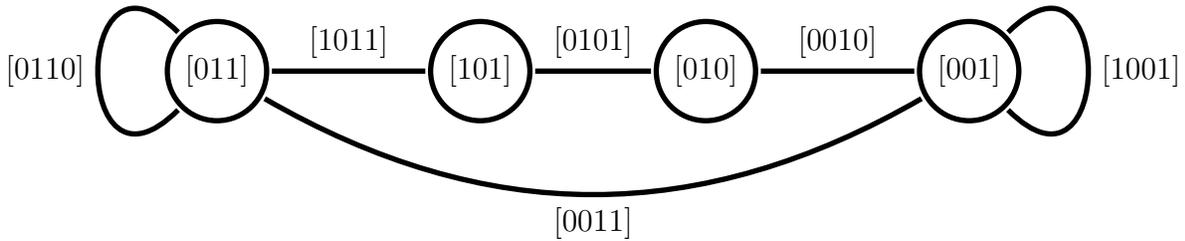

Now, we define the most important notion of the article, namely the $G$-richness. 
As  mentioned in Introduction, the classical richness has several equivalent characterizations, each of them is a candidate for a definition of the new notion.
Nevertheless, some of these characterizations are accompanied by technical complications when the reversal mapping $\Tr$ is replaced by a larger group $G$.
For example, the inequality in \Cref{nerovnostProVice} is valid only for $n$ which is $G$-distinguishing and such $n$ could be quite large. 
But in the case $G=\{ \id, \Tr \}$ the inequality is valid  for all nonnegative integers $n$ and therefore it can be used for equivalent definition of classical richness.

We have decided in \cite{PeSta1} to adopt definition of $G$-richness which is based on the notion of graphs of symmetries.
This specific tree structure of these graphs seems to be essential when considering palindromic richness.

\begin{definition}[\cite{PeSta1}] \label{def:G_and_almost_G_rich}
 We say
that $\uu$ is \emph{$G$-rich} if $\uu$ has Property $G$-tls$(1)$  and
$\uu$    is
 \emph{almost $G$-rich} if there exists $N \in \mathbb{N}$ such that $\uu$ has Property
 $G$-tls$(N)$.
\end{definition}

In the next section, we explain legitimacy  of  the name
$G$-richness, i.e., we show that the classical richness is
contained in our new definition of richness as well.

%======================================================

\subsection{Palindromic richness in the classical sense}\label{sec:clasika}

Let us recall the origin of palindromic richness.  In this section,
we use the word palindrome for $R$-palindrome and we denote by $\Pal(w)$  the set  of all
palindromic factors of a finite word $w$ including the empty word
$\varepsilon$. In \cite{DrJuPi}, Droubay, Justin and Pirillo
provided the following simple upper bound
\begin{equation}\label{eq:horni_mez_poctu_pali}
\# \Pal(w) \leq |w| + 1.
\end{equation}
 This bound serves for the definition of palindromic richness in the classical sense (see \cite{DrJuPi,GlJuWiZa}).
A finite word $w$ is \emph{rich} if $\# \Pal(w) = |w| + 1$.
An infinite word is \emph{rich} if all its factors are rich.
Another type of bound on the number of palindromes contained in an infinite word
was proved in \cite{BaMaPe}: if an infinite word has its language
closed under the reversal mapping, then the following inequality
holds
\begin{equation} \label{eq:nerovnost}
 \Delta  \C (n) + 2 \geq \P(n) + \P(n+1) \quad \text{   for all }
 n\in \mathbb{N}.
\end{equation}
The authors of \cite{GlJuWiZa} showed that an infinite word
$\uu$ with language closed under reversal is rich if and only if
the equality in \eqref{eq:nerovnost} is attained for all $n \in
\mathbb{N}$. Their proof uses the already mentioned notion of super reduced Rauzy
graph, which, in our terminology, is  the graph of symmetries
$\Gamma_n(\uu)$ for the group $G = \{ \id, \Tr \}$; a recurrent word $\uu$
is rich if and only if $\uu$ satisfies  -  again in our
terminology - Property $G$-tls$(1)$.

\begin{example}
All episturmian words (see \cite{DrJuPi}), which include Sturmian and Arnoux-Rauzy words, are rich.
Since the Fibonacci word $\uu_F$ is Sturmian, it is rich.
Other class of rich words are the words coding interval exchange transformations determined by a symmetric permutation, see \cite{BaMaPe}.
\end{example}

The following theorem summarizes properties characterizing rich recurrent words, their proofs can be found in
\cite{GlJuWiZa,DrJuPi, BuLuGlZa,BaPeSta2}.

\begin{thm}\label{equiv_rich}
For an infinite word $\uu$ with language closed under reversal
the following statements are equivalent:
\begin{enumerate}
\item  \label{equiv_rich_1} $\uu$ is rich,
 \item  \label{equiv_rich_returnpalindromicky} \emph{\cite{GlJuWiZa}} any complete return word of
any palindromic factor of $\uu$ is a~palindrome,
 \item \label{equiv_rich_returnobecny} \emph{\cite{GlJuWiZa}} for any
factor $w$ of $\uu$, every factor of $\uu$ that contains $w$ only
as its prefix and $R(w)$ only as its suffix is a~palindrome,
 \item \label{equiv_rich_suffix} \emph{\cite{DrJuPi, GlJuWiZa}} the  longest palindromic suffix of any factor $w \in \L(\uu)$ is unioccurrent in $w$,
 \item \label{equiv_rich_nasa} \emph{\cite{BuLuGlZa}} for
each $n \in \N$ the following equality holds $$\C(n+1) - \C(n) + 2 =
\P(n) + \P (n+1),$$
 \item \label{equiv_rich_7} \emph{\cite{BuLuGlZa}}  each  graph of symmetries
$\Gamma_n(\uu)$ satisfies: all its loops are palindromes and
the graph obtained from $\Gamma_n(\uu)$ by removing loops is a tree,
 \item  \label{equiv_rich_bilateral} \emph{\cite{BaPeSta2}} any bispecial factor $w$ of $\uu$
satisfies:
\begin{itemize}
\item if $w$ is non-palindromic, then
$$\b(w)= 0;$$
\item if $w$ is a~palindrome, then
$$\b(w)= \# \Pext(w) - 1.$$
\end{itemize}

\end{enumerate}
\end{thm}

Another characterization of rich words, which is not treated in this article, can be found in \cite{BuLuGlZa2}.

Richness in the classical sense is closely related to the notion of defect of a word.
As introduced in \cite{BrHaNiRe}, the \textit{defect of a finite word} $w$ is defined as follows
$$D(w) = |w| +1 - \# \Pal(w).$$
The \textit{defect of an infinite word} $\uu$ is defined as
$$D(\uu) = \sup_{w \in \Lu} \{ D(w) \}.$$
Recall that a word is rich if its defect is zero.

In \cite{BrHaNiRe} and \cite{BlBrGaLa}, the authors study also words with finite defect;
in \cite{GlJuWiZa},  such words are called \emph{almost rich}.
Words almost rich in the classical sense  can be characterized by properties analogous to those listed in \Cref{equiv_rich}.
For more details see \cite{BaPeSta3}.

The new definition is in fact based on a generalization of characterization \ref{equiv_rich_7} in \Cref{equiv_rich}.
The goal of this article is to generalize some of those characterizations:
\Cref{CharakteristikaReturn} generalizes characterization \ref{equiv_rich_returnpalindromicky}
and \Cref{CharakteristikaLongestSuffix} generalizes characterization \ref{equiv_rich_suffix}.
\Cref{G_rich_Pext_rovnosti} is a generalization of characterization \ref{equiv_rich_bilateral} for almost rich words.
Characterization 5 is already generalized in \cite{PeSta1}, we recall it here as \Cref{almostG-Rich}.
\Cref{sec:Defect} generalizes the notion of defect.

Let us point out a drawback of our new definition. The Property
$G$-tls$(1)$, unlike  the classical definition
of richness,  requires the language of an infinite word to be closed under reversal.
Nevertheless, an infinite word can be rich in the classical sense
without having its language closed under reversal. On the other
hand, as proved in \cite{GlJuWiZa} (Proposition 2.11), any
recurrent rich word has its language closed under reversal.
Therefore, on the set of recurrent words  both definitions
coincide.

Some generalizations were already made for groups $G = \{ \id, \Theta \}$.
As already mentioned, in the articles \cite{KaMa2008,LuLu},
the  reversal mapping $R$ is replaced by  an arbitrary involutive
antimorphism $\Theta$ and $\Theta$-palindrome is defined as a word $v$
satisfying $\Theta(v)=v$. Let us denote
 by $\PalT(w)$
the set of $\Theta$-palindromic factors occurring in $w$. As shown
in \cite{Sta2010},
\begin{equation}\label{eq:horni_mez_poctu_Tpali}
\# \PalT(w) \leq |w| + 1 - \gT(w),
\end{equation}
where $\gT(w) := \# \left \{ \{ a, \Theta(a)\} \mid a \in \A, a
\text{ occurs in } w \text{ and } a \neq \Theta(a) \right \}$.
Analogously to the classical richness, $\Theta$-richness is
introduced in \cite{Sta2010} as follows.  A finite word $w$ is \emph{$\Theta$-rich} if the
equality in \eqref{eq:horni_mez_poctu_Tpali} holds. An infinite
word is \emph{$\Theta$-rich} if all its factors are $\Theta$-rich.

Generalizations of some characterizations in \Cref{equiv_rich} for $\Theta$-rich
words are presented in \cite{Sta2010} as well.

\begin{remark}
In \cite{BlBrGaLa}, the authors remark that on binary alphabet the only periodic $E$-rich words are of period $2$. On the alphabet $\{ 0, 1\}$, the only finite $E$-rich words are the following:
$$
(01)^n, \ (10)^n, \ (01)^n0, \text{ and } (10)^n1
$$
for some $n \in \N$.

As stated in \cite{BrHaNiRe} for the reversal mapping $\Tr$, if a periodic word $w^\omega$ is closed under an involutive antimorphism $\Theta$, then $w$ can be written as a concatenation of two $\Theta$-palindromes.
This gives a restrictive (and only necessary) condition for periodic $G$-rich words for a larger group $G$.
Examples known to satisfy this condition are periodic generalized Thue-Morse words (see \cite{Sta2011}): words
$\tt_{b,m}$ defined in \Cref{ex:genTM} for $b \equiv 1 \pmod{m}$. These words are closed under the group $I_2(m)$.
\end{remark}

\section{Tools for characterization  of $G$-rich words}

Classical rich and almost rich words can be described using the
notions of return words and longest palindromic suffix. In order
to find a suitable description of $G$-richness, we first
introduce
 $G$-analogies of these notions and in the sequel we
demonstrate their efficiency. Let us recall that in all
definitions and statements in the sequel,  the symbol $G$ stands
for a finite subgroup of $AM (\mathcal{A}^*)$ such that it contains at
least one antimorphism.

\begin{definition}
A word $w \in \mathcal{A}^*$ is said to be \mbox{\emph{$G$-palindrome}}
if there exists an antimorphism $\Theta \in G$ such that
$w=\Theta(w)$.
\end{definition}

\begin{remark}
As already mentioned, if a word $w$ contains all letters of $\A$ and $\Theta(w) = w$ for an antimorphism $\Theta \in G$, then $\Theta$ is an involution.
However, non-involutive antimorphisms of $G$ also contribute to the number of generalized palindromes.
If $N$ is $G$-distinguishable and $w$ is a $\Theta$-palindrome such that $|w| \geq N$, then for any antimorphism $\Psi$ in $G$ the word $\Psi(w)$ is a $G$-palindrome, namely a $(\Psi \Theta \Psi^{-1})$-palindrome.
%Moreover, if $\Psi$ is a non-involutive antimorphism, then it can still have fixed points since
\end{remark}

\begin{definition} Let $w, v \in \mathcal{A}^*$.
\emph{$G$-occurrence} of a word $w$ in a word $v$ is an index $i$ such
that there exists $w' \in [w]$ having occurrence $i$ in $v$.

We say that $w$ is \emph{$G$-unioccurrent} in $v$ if $w$ occurs in $v$
and there is no other \mbox{$G$-occurrence} of $w$ in $v$.

\end{definition}

\begin{definition}
Let $\uu  \in \mathcal{A}^\mathbb{N}$ be an infinite word  and $w
\in \L(\uu)$. A factor $v \in \L(\uu)$ of length $|v| >|w|$ is
called \emph{complete $G$-return word of $[w]$ in $\uu$} if
\begin{itemize}
\item a prefix and a suffix of $v$ belong to $[w]$ and
\item $v$ contains no other $G$-occurrence of $w$.
\end{itemize}
%If, moreover, the  word $v$ has a prefix  $w$, then $v$ is called complete $G$-return word of $w$ in $\uu$.

\vspace{\baselineskip}

We say that $v' \in \Lu$ is a \emph{$G$-return word} of $[w]$ in $\uu$ if
for some $w' \in [w]$ the word $v'w'$ is a complete $G$-return
word of $[w]$.
%Definition of $G$-return word of $w$ is analogous.

\end{definition}

\begin{definition}\label{de:Glps}
A suffix $w$  of a word $v \in \mathcal{A}^*$ is called
\emph{$G$-longest palindromic suffix} of $v$  if
\begin{itemize}
\item $w$ is a $G$-palindrome and
\item $|w| \geq |w'|$ for any
$G$-palindromic suffix $w'$ of $v$.
\end{itemize}

\noindent The $G$-longest palindromic suffix of $v$  is denoted by $G$-${\rm lps}(v)$.

\end{definition}

\begin{remark}  
If no nonempty suffix $w$ of a nonempty word $v\in \mathcal{A}^*$ is a $G$-palindrome, then $G$-${\rm lps}(v) = \varepsilon$.
(Let us recall that any index $i\in \{1,2,\ldots,n\}$ is defined to be an occurrence of $\varepsilon$ in  $v = v_1v_2\cdots v_{n-1}$.)
The shortest example of such phenomenon is a one-letter word $a \in \A$ and a group $G$ containing no antimorphism that fixes the letter $a$.
It may occur only in the case when $G$ does not contain the reversal mapping $\Tr$. 
Clearly, the converse is not true: if $R \not \in G$, then there may exist an antimorphism in $G$ that fixes the last letter of $w$.
 \end{remark}

Let us demonstrate these definitions on an example.

\begin{example}
Take again the Thue-Morse word and the group $I_2(2) = \{\id, \Tr, E, E \Tr \}$.
The word $011$ is a factor of the word, we have
$
[011] = \{011, 110, 100, 001\}.
$
All $I_2(2)$-occurrences of $011$ in the prefix $p = 01101001100$ of the Thue-Morse word form the following set of indices:
$$
\{0, 1, 4, 5, 6, 7, 8\}.
$$
The factor $001100$ is $I_2(2)$-unioccurrent in $p$.
The following $I_2(2)$-complete return words of $[011]$ are contained in $p$:
$$
\{0110, 110100, 1001, 0011, 0110, 1100 \}.
$$

The $I_2(2)$-longest palindromic suffix of $p$ is
$$
I_2(2)\text{-lps}(p) = 001100.
$$

%Figure \ref{fig:example-TM-Rauzy} shows the Rauzy graph of order $3$ of the Thue-Morse word.
%Figure \ref{fig:tm_example_directed} depicts the graph $\overrightarrow{\Gamma}_3(\tt_{2,2})$, while Figure \ref{fig:tm_example_undirected} shows the graph ${\Gamma}_3(\tt_{2,2})$.

\end{example}

%======================================================

\section{$G$-richness and $G$-return words}

In this section, namely in \Cref{CharakteristikaReturn}, we demonstrate  that the notions of complete $G$-return word and $G$-palindrome can grasp the essence of $G$-richness.
We prove that if $\uu$ is an infinite word with language closed under $G$, then
$\uu$ is $G$-rich if and only if for all $w \in \Lu$ every complete $G$-return word of $[w]$ is a $G$-palindrome.
The theorem generalizes characterization \ref{equiv_rich_returnpalindromicky} in \Cref{equiv_rich}.

%Let us recall that according to \Cref{def:G_and_almost_G_rich}, almost $G$-richness means property $\Gg{G}{N}$ for some $N$. 

To describe the generalized characterization of almost $G$-richness, we introduce the following property, called Property $\Gc{G}{N}$.
Again, the name of the property contains the two parameters $G$ and $N$ so that we can easily keep track of them.
The abbreviation ``crw'' stands for complete return word, since the property is based on complete $G$-return words.

% First we introduce a new property  based on $G$-return words which depends on a parameter $N$ as well.

\begin{definition} Let $N \in \N$. We say that $\uu \in \mathcal{A}^\mathbb{N}$ satisfies \emph{Property $\Gc{G}{N}$} if for all $w \in \Lu$, $|w| \geq N$, every complete
$G$-return word of $[w]$ is a $G$-palindrome.
\end{definition}

Before proving the main result of this section we introduce several lemmas.

\begin{lem} \label{lem:rich_na_crw} Let $N \in \N$ and $\uu \in \mathcal{A}^\mathbb{N}$  satisfy Property $\Gg{G}{N}$.
If $w$ is a factor of $\uu$ such that $|w| \geq N$ and $v$  be a complete $G$-return word of $[w]$ in
$ \uu$ starting in $w$,
then there exist a letter $a \in
\mathcal{A}$ and an antimorphism $\Theta \in G$  such  that $wa$
is a prefix of $v$ and $\Theta(a) \Theta(w)$ is a suffix of $v$.
\end{lem}

Note that we do not assume explicitly that $\Lu$ is closed under $G$.
However, this property is satisfied since it is part of the definition of Property $\Gg{G}{N}$.

\begin{proof}
Take $n \geq N$, $w \in \L_n(\uu)$, and $v$ a complete $G$-return
word of $[w]$ starting in $w$.
Denote $w'\in [w]$ the suffix of $v$ of length $n$.

At first we suppose that $w$ is a special factor of $\uu$.
We consider the following two cases.
We exploit the graphs of symmetries $\overrightarrow{\Gamma}_n(\uu) =
(V,\overrightarrow{E})$ and ${\Gamma}_n(\uu) = (V,{E})$.

\begin{enumerate}

\item  If besides the prefix  $w$ and the suffix $w'$ the complete
$G$-return word $v$ contains no other occurrences of  a special factor
of length $n$, then $[v]$ is an edge in $E$ which starts and ends
in the same vertex $[w]$.
 The edge $[v]$ is thus a loop and according to the definition  of Property
$\Gg{G}{N}$, the factor $v$ is a $G$-palindrome. Obviously, $v$
has the property stated in the claim.

\item Let $v = v_0v_1 \cdots v_m$ contain a special factor $z
\notin [w]$ of length $n$ at the position $i$, i.e., $z = v_i \cdots v_{i+n-1}$. We may suppose
without loss of generality that $i$ is the least index with
this property. Then $[z]$ is a vertex of
$\overrightarrow{\Gamma}_n(\uu)$ and a prefix of $v$ is an edge in
$\overrightarrow{\Gamma}_n(\uu)$ starting in $[w]$ and ending in
$[z]$. Since the graph obtained from $\Gamma_n(\uu)$ by removing
loops is a tree, the complete $G$-return word $v$ has a suffix $f
\in \overrightarrow{E}$ such that a prefix of  $f$ belongs to
$[z]$ and its suffix belongs to $[w]$. Moreover, there exists
an antimorphism $\Theta$ such that $f = \Theta(e)$. As $|e| = |f|
> |w|$,  the factor $v$ has the property stated in the claim.
\end{enumerate}

Let us now suppose that $w$ is  not a  special factor and thus $w$
has a unique right extension, say $a$. If $w'=\Theta(w)$ for some
antimorphism $\Theta\in G$, then $\Theta(w)$ has a unique left
extension $\Theta(a)$ and therefore    $\Theta(a)\Theta(w)$ is a
suffix of the complete $G$-return word $v$ as stated in the claim.

To finish the proof, it is enough to consider the situation when $w' = \mu(w)$
for some morphism $\mu\in G$  and $w$ is not a special factor.
We discuss two separate cases.

\begin{enumerate}
\item There exists no special factor of length $n$.\\
In this case $\uu$ is periodic. Denote $v'$ the word such that $v
= v'\mu(w)$. Because no special factor  of length at least $n$
exists, the factor $v=v'\mu(w)$ is the unique right prolongation of $w$
of length $|v|$. As $\Lu$ is closed under $G$, $\mu^\ell(v) =
\mu^\ell(v')\mu^{\ell+1}(w)$ is the unique right prolongation of
$\mu^{\ell}(w)$ of given length. In particular, for $\ell = 1$ it
implies  $v'\mu(v')\mu^2(w) \in \L(\uu)$. Repeating this argument
for $\ell =2,3,\ldots$ we deduce that $v'\mu(v')\mu^2(v') \cdots
\mu^{\ell}(v')\mu^{\ell +1}(w) \in\L(\uu)$. Therefore, the factor
$v'\mu(v')\mu^2(v') \cdots \mu^{k-1}(v')$, where  $k$ is the order
of the morphism $\mu$, is a period of $\uu$ which does not
contain any antimorphic image of $w$ - a contradiction.

\item There exists a special factor of length $n$.\\
Consequently, there exists a unique $q$ such that $wq$ is right special and no proper
prefix of $wq$ is right special.  The factor $wq$ has only one occurrence of
a factor $\nu(w)$ for some morphism $\nu\in G$ - in the opposite case,
we can find a shorter prolongation of $w$ which is right special.  Since $v$
has suffix $\mu(w)$, we deduce  $|wq| < |v|$.  As  $\mu$ is a
morphism, $\mu(w)\mu(q)$ is the only right prolongation of $\mu(w)$
and thus $v\mu(q)$ is a complete $G$-return word of $[wq]$.  For the special
factor  $wq$, we may now use  the first part of the proof and thus
 find an antimorphism $\Theta$ such that $\mu(w)\mu(q) = \Theta
(wq) = \Theta(q)\Theta(w)$. Applying the morphism $\mu^{-1}$, we
get $wq=\mu^{-1} \Theta (q) \mu^{-1} \Theta (w)$. Together with
the inequality $|wq| < |v|$, it implies a contradiction with the fact
that $v$ is a complete $G$-return word of $[w]$. \qedhere
\end{enumerate}
\end{proof}

\begin{lem}\label{GraphImplikujeCRet}
Let $\uu \in \mathcal{A}^\N$ and $N \in \N$. If $\uu$ has Property $\Gg{G}{N}$, then it has Property $\Gc{G}{N}$.
\end{lem}

\begin{proof}
Let $v$  be a complete $G$-return word of $[w]$ starting in $w$
for a factor $w$ such that $|w| \geq N$.
Denote by $a$ a letter such that $wa$ is a prefix of $v$.

If $wa = v$, then \Cref{lem:rich_na_crw} implies that $v$ is a
$G$-palindrome.

If $wa \neq v$, then according to \Cref{lem:rich_na_crw}, $v$ is a
complete $G$-return word of $[wa]$ as well. We apply the procedure again on
$wa$. We find a letter $b$ such that $wab$ is a prefix of $v$. If
$wab= v$, the $v$ is a $G$-palindrome, otherwise $v$ is a complete
$G$-return word of $[wab]$.

We continue in this way until the procedure stops and we conclude
that $v$ is a \mbox{$G$-palindrome}.
\end{proof}

\begin{remark} \label{rem:alternuji}
As a consequence of the previous lemma, we have that
for an infinite word $\uu$ satisfying Property $\Gg{G}{N}$,  the occurrences
of morphic and antimorphic images of a factor $w$ satisfying $|w| \geq N$ alternate.
This consequence of the previous claim is an analogy to the claims
stated in \cite{GlJuWiZa} for rich words and in \cite{Sta2010} for
$\Theta$-rich words: given a $\Theta$-rich word for an involutive antimorphism $\Theta$, the occurrences of $w$ and $\Theta(w)$ in the word alternate.
\end{remark}

\begin{thm} \label{thm:Ggenerovane}
If there exists an almost $G$-rich word, then $G$ is generated by the set of its involutive antimorphisms.
\end{thm}

\begin{proof} 
Let $\uu$ be an almost $G$-rich word.
Let $N$ be an integer such that $\uu$ satisfies the property $G$-tls($N$).
Let $w$ be a factor of length at least $N$ and such that all letters occur in it.
According to \Cref{lem:rich_na_crw}, any
occurrence of a word from $[w]$ is a $\Theta$-image of the left
closest  occurrence of a factor from $[w]$ for some antimorphism $\Theta \in G$.
According to \Cref{GraphImplikujeCRet}, all complete $G$-return words of $[w]$ are $G$-palindromes.
Thus, since all letters occur in $w$, such antimorphism $\Theta$ is involutive. 
Therefore, for any $\nu \in G$, the factor $\nu(w)$
occurring in  $\uu$ can be written as $\nu(w) = \Theta_1\Theta_2
\cdots \Theta_s(w)$, where $\Theta_1\Theta_2 \cdots \Theta_s$ is a
sequence of involutive antimorphisms.
Since $w$ contains all letters, the number $|w|$ is $G$-distinguishing, and thus according to \Cref{def:Gdistinguishing} the equality $\nu(w) = \Theta_1\Theta_2 \cdots \Theta_s(w)$ implies $\nu = \Theta_1\Theta_2 \cdots \Theta_s$.

In other words, any element $\nu$ of the group $G$ can be written as a composition of involutive antimorphisms, i.e.,
the group is generated by involutive elements.
\end{proof}

%TODO: poznamka, ze nemusi byt abelovske

The following lemma is the converse of \Cref{GraphImplikujeCRet}.
However, we need to add an explicit assumption of closedness under $G$.
In \Cref{GraphImplikujeCRet}, this assumption is hidden in the definition of Property $\Gg{G}{N}$,
which includes it (unlike Property $\Gc{G}{N}$).

\begin{lem}\label{CRetImplikujeGraph}
Let $\uu$ be an infinite word with language closed under $G$ and $N \in \N$.
If $\uu$ satisfies Property $\Gc{G}{N}$, then it satisfies Property $\Gg{G}{N}$.
\end{lem}

\begin{proof}
Let $n \geq N$ and $w \in \L_n(\uu)$. We assume that every
complete $G$-return word of $[w]$ is a $G$-palindrome. We have to show two
properties of $\Gamma_n(\uu)$.

\begin{enumerate}
\item
\emph{Any loop in $\Gamma_n(\uu)$ is a $G$-palindrome:} \\
Since any loop  $e$ in $\Gamma_n(\uu)$  at a vertex $[w]$  is a
complete $G$-return word of $[w]$, the loop $e$ is a $G$-palindrome by our
assumption.

\item
\emph{The graph obtained from $\Gamma_n(\uu)$ by removing loops is a tree:} \\
Or equivalently, we  show that in $\Gamma_n(\uu)$ there exists
unique path between  any two different vertices $[w']$ and
$[w'']$.   Let $p$ be a factor of $\uu$ such that a prefix of $p$
belongs to $[w']$, its suffix belongs to $[w'']$ and $p$ has no
other occurrences of factor from $[w']$ or $[w'']$. Let without
loss of generality $w'$ be a prefix of $p$. Let us find   a
complete $G$-return word of $[w']$ with prefix $p$, denote it $v$.
Since  $v$ is a $G$-palindrome, the factor $\Theta(p)$ is a suffix
of $v$ for some antimorphism $\Theta \in G$. As $v$  is a complete
$G$-return word of $[w']$, $v$ has exactly two $G$-occurrences of $w'$.
The factor $v$ contains at least two
$G$-occurrences of $w''$. Therefore, the next factor
with the same properties as $p$, i.e., representing a path in the
undirected graph $\Gamma_n(\uu)$ between $[w']$ and $[w'']$, which
occurs in $\uu$ after $p$, is $\Theta(p)$. Consequently, any factor
with the same properties as $p$ belongs to the same equivalence
class $[p]$. \qedhere
\end{enumerate}
\end{proof}

\begin{thm} \label{CharakteristikaReturn} 
If $\uu$ is an infinite word with language closed under $G$, then
\begin{enumerate}
\item $\uu$ is $G$-rich if and only if for all $w \in \Lu$ every complete $G$-return word of $[w]$ is a $G$-palindrome, i.e., $\uu$ has Property $\Gc{G}{1}$;
\item $\uu$ is almost $G$-rich if and only if there exists and integer $N$ such that for all $w \in \Lu$ longer than $N$ every complete $G$-return word of $[w]$ is a $G$-palindrome, i.e., $\uu$ has Property $\Gc{G}{N}$.
\end{enumerate}
\end{thm}

\begin{proof}
\Cref{GraphImplikujeCRet,CRetImplikujeGraph} together state that if $N$ is an integer,
then $\uu$ satisfies $\Gg{G}{N}$ if and only if it satisfies $\Gc{G}{N}$.
The second claim then directly follows from the definition of almost $G$-richness and Property $\Gc{G}{N}$.
The first claim is obtained if $N = 1$.
\end{proof}

%======================================================

\section{$G$-richness and $G$-longest palindromic suffix}

As we already stated, the classical richness  is connected to the number of occurrences of the longest palindromic suffix in  any factor.
This section aims to generalize this connection which is given by characterization \ref{equiv_rich_suffix} in \Cref{equiv_rich}.
The main result of this section is in \Cref{CharakteristikaLongestSuffix}.

In the case of classical palindromes, the longest palindromic suffix of a nonempty word is always  nonempty, but it  is not always satisfied for the $G$-longest palindromic suffix.
Therefore, the characterization of $G$-richness by the $G$-longest palindromic suffix needs a modification.
In this section we show that if $\uu \in \mathcal{A}^\mathbb{N}$ is an infinite
word with language closed under $G$, then  $\uu$ is $G$-rich if and only if for
any factor  $v \in \Lu$, its  $G$-longest palindromic suffix
 is $G$-unioccurrent in $v$ or the last letter
of $v$ is $G$-unioccurrent in $v$.

To describe the needed property we introduce the next definition of Property $\Gl{G}{N}$,
again with $G$ and $N$ as parameters and the abbreviation ``lps'' standing for longest palindromic suffix.

\begin{definition}\label{def:PropGlps} Let $N \in \N$.
 We say that $\uu \in \mathcal{A}^\mathbb{N}$ satisfies
\emph{Property $\Gl{G}{N}$} if for all $w \in \Lu$, $|w| \geq N$, either
the word $G$-${\rm lps}(w)$ is $G$-unioccurrent in $w$,
 or the suffix of $w$ of length $1$ has exactly one $G$-occurrence in $w$.
\end{definition}

\begin{lem}\label{lem:crwAlps}
 Let  $\uu \in \mathcal{A}^\mathbb{N}$.
\begin{enumerate} \item  If  $\uu$ has Property $\Gc{G}{1}$, then  $\uu$ has Property $\Gl{G}{1}$.
\item  Let $N \in \N$. If $\uu$ is uniformly recurrent and has Property
$\Gc{G}{N}$, then there exists $M \in \mathbb{N}$ such that $\uu$
has Property $\Gl{G}{M}$.
\end{enumerate}
\end{lem}

\begin{proof}   Let us realize a trivial fact:  if  $w \in \L(\uu)$  has a suffix $v$ which is not
$G$-unioccurrent in $w$, then there is a suffix of $w$ which is a complete $G$-return word of $[v]$.

 To prove the first assertion, consider a factor $w \in \L(\uu)$. If the last letter of $w$, denoted $a$, is
\mbox{$G$-unioccurrent}, we have nothing to do. If $a$
is not $G$-unioccurrent in $w$, then according to the mentioned
fact and by Property  $\Gc{G}{1}$, a complete $G$-return word of $[a]$ is a $G$-palindrome of length greater than $1$. Therefore,
$G$-${\rm lps}(w) \neq \varepsilon$. We have to show that
$G$-${\rm lps}(w) $ is $G$-unioccurrent in $w$. If not,  then a
suffix of $w$ is  a complete $G$-return word of $[G$-${\rm lps}(w)]$
which, according to Property $\Gc{G}{1}$, is a \mbox{$G$-palindrome} longer than the
$G$-longest palindromic suffix of $w$  - a contradiction.

Now we prove the second assertion. Since $\uu$ is uniformly
recurrent, there exists an integer $M$ such that
 every factor $w \in \Lu$, $|w| \geq M$, contains at least two
 occurrences  of every factor $v \in \L_{N}(\uu)$. It particular, it implies that
 $w$ contains at least two \mbox{$G$-occurrences} of its suffix $z$ of
 length $N$. By the fact mentioned at the beginning of the proof,
Property $\Gc{G}{N}$ implies that $G$-${\rm lps}(w)$ is longer
than $N$  and is $G$-unioccurrent in $w$.
\end{proof}

\begin{remark} Let us note that the second part of the
previous lemma can be proved considering a  weaker assumption than
uniform recurrence of $\uu$. It is enough to assume that
 every factor $w \in \Lu$ has  only finitely many
complete $G$-return words.
\end{remark}

\begin{lem}\label{lem:lpsAcrw}
 Let  $\uu \in \mathcal{A}^\mathbb{N}$.
\begin{enumerate} \item  If  $\uu$ has Property  $\Gl{G}{1}$,  then  $\uu$ has Property $\Gc{G}{1}$.
\item  If $\uu$  has Property $\Gl{G}{N}$ such that  $N>1$, then $\uu$
has Property $\Gc{G}{N-1}$.
\end{enumerate}
\end{lem}

\begin{proof}
We prove both claims simultaneously. If $N = 1$, set $M = 1$.
Otherwise set $M = N - 1$.
We prove by contradiction that $\uu$ satisfies Property $\Gc{G}{M}$.

Suppose there is a factor $w \in \Lu$, $|w| \geq M$, such that
there is a factor $v \in \Lu$ which is a complete $G$-return word
of $[w]$ and is not a $G$-palindrome.
Denote $p$ the prefix of $\uu$ ending in the leftmost occurrence
of a factor from $[v]$. It is clear that $|p| \geq N$.

Since $p$ ends in a nonempty $G$-palindromic  complete $G$-return
word, the suffix of $p$ of length $1$ has at least two
$G$-occurrences
 and thus Property $\Gl{G}{N}$ assures that $p$ has a nonempty $G$-longest palindromic suffix which is $G$-unioccurrent.
Let us denote $x := G\text{-}{\rm lps} (p)$.

If $0 < |x| \leq |w|$, then $x$ has at least two $G$-occurrences
in $p$ - a contradiction.

If $|w| < |x| < |v|$, then we can find a third $G$-occurrence of
$w$ in $v$ - a contradiction with $v$ being a complete $G$-return
word of $[w]$.

If $|x| = |v|$, then we have a contradiction with $v$ not being a
$G$-palindrome.

If $|x| > |v|$, then we can find a factor $v' \in [v]$ such that its occurrence is in contradiction with the choice of the prefix $p$.
\end{proof}

\begin{remark}  Again, the assumptions of the previous lemma can be
reduced, as it is visible in our proof. It is enough to require
that any prefix $v$  of $\uu$ of length greater than or equal to $N$ has unique $G$-longest palindromic suffix or the last letter of the prefix $v$ is
$G$-unioccurrent  in $v$.

\end{remark}

\begin{thm}\label{CharakteristikaLongestSuffix} 
Let $\uu \in \mathcal{A}^\mathbb{N}$ be an infinite word with language closed under $G$.
\begin{enumerate}
\item The word $\uu$ is $G$-rich if and only if for any factor  $v \in \Lu$, its  $G$-longest palindromic suffix  is $G$-unioccurrent in $v$ or the last letter of $v$ is $G$-unioccurrent in $v$, i.e., $\uu$ has Property $\Gl{G}{1}$.
\item If $\uu$ is uniformly recurrent, then $\uu$ is almost $G$-rich if and only if there exists an integer $N$ such that for any factor $v \in \Lu$ longer than $N$, its  $G$-longest palindromic suffix  is $G$-unioccurrent in $v$ or the last letter of $v$ is $G$-unioccurrent in $v$, i.e., $\uu$ has Property $\Gl{G}{N}$.
\end{enumerate}
\end{thm}

\begin{proof}
Using \Cref{lem:crwAlps,lem:lpsAcrw}, we have that Property $\Gl{G}{1}$ is satisfied if and only if Property $\Gc{G}{1}$ is satisfied. The first claim then follows from the first claim of \Cref{CharakteristikaReturn}.

If $\uu$ is uniformly recurrent, then again using \Cref{lem:crwAlps,lem:lpsAcrw}, we find that
there exists an integer $N$ such that $\uu$ satisfies Property $\Gl{G}{N}$ if and only if there exists an integer $M$ such that $\uu$ satisfies Property $\Gc{G}{M}$.
The second claim follows from the second claim of \Cref{CharakteristikaReturn}.
\end{proof}

%=============================================
%=============================================
%=============================================
%=============================================
%=============================================

\section{G-defect}\label{sec:Defect}

In \Cref{sec:clasika} we recalled the definition of palindromic defect and its relation to classical richness.
Moreover,  since the defect  of a finite word $w$ depends only on its length and the number of palindromes contained in it,
the defect satisfies the following properties for any $a \in \A$ (see \cite{BrHaNiRe}):
\begin{equation*} \label{vnoreni}
D(w) \leq  D(wa)\leq D(w)+1, \   D(w) \leq D(aw)\leq D(w)+1, \ \hbox{and} \ D(w) = D(R(w)).
\end{equation*}

The question we address here is how to define a $G$-analogy
of defect when the group $G$ contains more than two elements.
Of course, we would like to find a definition of $G$-defect
such that $G$-richness and almost $G$-richness are again connected
with $G$-defect in an analogous way.

 Let us illustrate on the Thue-Morse word $\tt_{2,2}$ the number of distinct $G$-palindromes contained in its factors.
 The language of the Thue-Morse word is invariant under the  reversal mapping $R$ and under the antimorphism $E$ which permutes letters $0$ and $1$.
  In \cite{PeSta1}, we showed that the Thue-Morse word is
  $G$-rich  for $G= \{ \id, \Tr, E, \Tr E\}$. In \Cref{tab:TM_prefixy}, the numbers of
 \mbox{$G$-palindromic} factors of short prefixes of the Thue-Morse words
 are depicted. 
There is no simple relation between the number of palindromes, $E$-palindromes and the length of the prefix,
nevertheless, the $G$-longest palindromic suffix of each prefix is $G$-unioccurrent in it.
To generalize the notion of defect, the counting of $G$-palindromes must be replaced by counting the classes $[w]$ of $G$-palindromes.
Thus, we define the set, denoted $\PalG(w)$, of all $G$-palindromic classes of equivalence in a finite word $w$ as follows
$$
\PalG(w) := \left \{ [v] \mid v \text{ is a factor of } w \text{ and a } G \text{-palindrome}  \right \} .
$$

 \renewcommand{\tabcolsep}{0.1cm}
 \renewcommand{\arraystretch}{1.2}

 \begin{table}[t!]
 \begin{center}
 \begin{tabular}{c|c|c|c||c|c|c|c}
 $n$ & $\# \Pal_{\Tr} (p_n)$ & $\# \Pal_E(p_n)$ & $G$-${\rm lps}(p_n)$ &
 $n$ & $\# \Pal_{\Tr} (p_n)$ & $\# \Pal_E(p_n)$ & $G$-${\rm lps}(p_n)$ \\ \hline
 $0$ & $1$ & $1$ & $\varepsilon$ &
 $10$ & $9$ & $8$ & $100110$ \\ \hline
 $1$ & $2$ & $1$ & $0$ &
 $11$ & $10$ & $9$ & $001100$ \\ \hline
 $2$ & $3$ & $2$ & $01$ &
 $12$ & $11$ & $10$ & $10011001$ \\ \hline
 $3$ & $4$ & $2$ & $11$ &
 $13$ & $12$ & $10$ & $0100110010$ \\ \hline
 $4$ & $5$ & $3$ & $0110$ &
 $14$ & $13$ & $11$ & $101001100101$ \\ \hline
 $5$ & $6$ & $3$ & $101$ &
 $15$ & $14$ & $12$ & $11010011001011$ \\ \hline
 $6$ & $7$ & $4$ & $1010$ &
 $16$ & $15$ & $13$ & $0110100110010110$ \\ \hline
 $7$ & $8$ & $5$ & $110100$ &
 $17$ & $16$ & $13$ & $101101$ \\ \hline
 $8$ & $9$ & $6$ & $01101001$ &
 $18$ & $17$ & $13$ & $01011010$ \\ \hline
 $9$ & $9$ & $7$ & $0011$ &
 $19$ & $18$ & $13$ & $0010110100$ \\ %\hline
 \end{tabular}
 \end{center}
 \caption{Count of palindromes and $E$-palindromes in the prefixes of the Thue-Morse word $\tt_{2,2}$. The prefix of $\tt_{2,2}$ of length $n$ is denoted $p_n$.}
 \label{tab:TM_prefixy}
 \end{table}

\begin{definition}\label{def:D_G_nove}
Let $w$ be a finite word.
The \emph{$G$-defect} of $w$ is defined as
$$
D_G(w) := |w| + 1 - \# \PalG(w) - \gamma_G(w),
$$
where
$$
\gamma_G(w) := \# \left \{ [a] \mid a \in \A, a \text{ occurs in } w \text{, and } a \neq \Theta(a)   \text{ and for every antimorphism }  \Theta \in G \right \}.
$$
\end{definition}

It follows from the definition that for all $w \in \A^*$ and $\mu \in G$ we have $D_G(w) = D_G(\mu(w))$.

The authors of  \cite{DrJuPi} also observed that the classical  richness of $w$
can be characterized by so-called Property \textit{Ju}:
{\it Any prefix of $w$ has unioccurrent longest palindromic
suffix.} The notion of the longest palindromic suffix helps to calculate the defect of a word.  For a word $w$ and a letter $a$, the following holds (see \cite{GlJuWiZa}):
$$D(wa) = \left\{\begin{array}{ll}
D(w)\,,& \hbox{if $wa$ has unioccurrent longest palindromic suffix}\\
D(w)+1\,,~~~& \hbox{otherwise. }\\
\end{array}\right.
$$
Therefore, the defect of a finite word $w=w_1w_2\cdots w_n$ equals
to the number of indices $i$ for which $w_1w_2\cdots w_{i-1}w_i$ does
not have a unioccurrent longest palindromic suffix. Such indices are called lacunas  in
\cite{BlBrGaLa} and defective positions in \cite{GlJuWiZa}.
Inspired by this, we adopt the following definition.

\begin{definition} \label{def:DefPos}
Let $w = w_1 \cdots w_n \in \A^*$. An integer $i$ such that $1 \leq i \leq n$ is
called \emph{$G$-lacuna} in $w$ if $w_i$ and $G\text{-}{\rm
lps}(w_1 \cdots w_i)$ are not $G$-unioccurrent in $w_1 \cdots
w_i$.
\end{definition}

The next lemma follows from comparing the last definition and the definition of \mbox{$G$-defect}.

\begin{lem} \label{lem:D_G_pozice}
Let $w \in \A^*$, then
$$
D_G(w) = \text{the number of }G\text{-lacunas in } w.
$$
\end{lem}

\begin{proof}
If $w$ is the empty word, than the claim holds.
Suppose $w = w_1\cdots w_i$ for some $i \geq 1$.
We will show the two following implications:
\begin{enumerate}
\item if $i$ is a $G$-lacuna, then $D_G(w) = D_G(w_1\cdots w_{i-1}) + 1$;
\item if $i$ is not a $G$-lacuna, then $D_G(w) = D_G(w_1\cdots w_{i-1})$.
\end{enumerate}

Denote $s = G\text{-}{\rm lps}(w)$.

Suppose $i$ is a $G$-lacuna, i.e., $w_i$ and $s$ are both not $G$-unioccurrent in $w$.
Since $w_i$ is not $G$-unioccurrent, we have $\gamma_G(w) = \gamma_G(w_1\cdots w_{i-1})$.
Since $s$ is not $G$-unioccurrent, we have $\# \Pal_G(w) = \# \Pal_G(w_1\cdots w_{i-1})$.
This shows the first implication.

Suppose $i$ is not a $G$-lacuna, i.e., $w_i$ is $G$-unioccurrent or $s$ is $G$-unioccurrent in $w$.
We distinguish the three following cases.
\begin{enumerate}[a)]
\item $w_i$ is $G$-unioccurrent and $s$ is not $G$-unioccurrent. \\ It follows that $s = \varepsilon$ and thus for every antimorphism $\Theta \in G$, we have that $\Theta(w_i) \neq w_i$, which implies that $\gamma_G(w) = \gamma_G(w_1\cdots w_{i-1}) + 1$ and $\# \Pal_G(w) = \# \Pal_G(w_1\cdots w_{i-1})$.
\item $w_i$ is not $G$-unioccurrent and $s$ is $G$-unioccurrent. \\ It this case, since $w_i$ is not $G$-unioccurrent, we have  $\gamma_G(w) = \gamma_G(w_1\cdots w_{i-1})$. $G$-unioccurrence of $s$ implies  $\Pal_G(w) = \Pal_G(w_1\cdots w_{i-1}) \cup \{s\}$, thus, $\# \Pal_G(w) = \# \Pal_G(w_1\cdots w_{i-1})+1$.
\item$w_i$ is $G$-unioccurrent and $s$ is $G$-unioccurrent. \\ $G$-unioccurrence of $w_i$ implies that $|s| < 2$. Since $i \geq 1$, $s \neq \varepsilon$. Thus $|s| = 1$ and we deduce that $\gamma_G(w) = \gamma_G(w_1\cdots w_{i-1})$ and $\# \Pal_G(w) = \# \Pal_G(w_1\cdots w_{i-1})+1$.
\end{enumerate}
In all three cases we conclude that $D_G(w) = D_G(w_1\cdots w_{i-1})$ which shows the second implication.
\end{proof}

Moreover, it can be easily shown that the following relations are preserved:
\begin{equation*} \label{D_G_vnoreni}
D_G(w) \leq  D_G(wa)\leq D_G(w)+1 \text{ and }   D_G(w) \leq D_G(aw)\leq D_G(w)+1
\end{equation*}
for all $w \in \A^*$ and $a \in \A$.
Therefore, we can define $G$-defect of an infinite word.

\begin{definition}
Let $\uu$ be an infinite word. The \emph{$G$-defect} of $\uu$, denoted
$D_G(\uu)$, is defined as
$$
D_G(\uu) := \sup_{w \in \Lu} \{ D_G(w) \}.
$$
\end{definition}

The immediate connection with Property $\Gl{G}{N}$ is summarized in the following lemma.

\begin{lem} \label{lem:D_G_a_G_lps}
Let $\uu$ be an infinite word with language closed under $G$.
\begin{enumerate}

\item  \label{lem:D_G_a_G_lps_1} $D_G(\uu)=0$ if and only if $\uu$ satisfies Property $\Gl{G}{1}$.

\item  \label{lem:D_G_a_G_lps_2}  If there exists an integer $N$ such that $\uu$ satisfies Property $\Gl{G}{N}$, then $D_G(\uu)$ is finite.

\item  \label{lem:D_G_a_G_lps_3} If $\uu$ is uniformly recurrent and  $D_G(\uu)$ is finite, then there exists an integer $N$ such that $\uu$ satisfies Property $\Gl{G}{N}$.
\end{enumerate}

\end{lem}

\begin{proof}
The first two claims follow from \Cref{lem:D_G_pozice}.

To show the last claim, suppose $D_G(\uu) $ to be finite. Then there exists a prefix $v=u_0u_1\cdots u_{M-1}$  of $\uu$ such that $D_G(\uu)  = D_G(v)$ and any letter of the alphabet occurs in $v$. As $\uu$ is uniformly recurrent, there exists a constant $N$ such that any factor $w$  of $\uu$ of length at least $N-1$ contains the prefix  $v$ as its factor. Using the definition and basic properties of $G$-defect, and maximality of $D_G(v)$, we obtain
  $ D_G(v)=D_G(w) =D_G(wa) $ for any $a \in \A$  such that $wa\in \Lu$. Therefore,  the last position in $wa$ is not a $G$-lacuna,  i.e.,  $\uu$  has Property $\Gl{G}{N}$.
\end{proof}

It remains to connect $G$-defect with $G$-richness and almost $G$-richness.

\begin{thm} \label{CharakteristikaDG}
Let $\uu$ be an infinite word with language closed under $G$.
\begin{enumerate}

\item   $D_G(\uu)=0$ if and only if $\uu$ is $G$-rich.

\item If $\uu$ is uniformly recurrent, then  $D_G(\uu)$ is finite  if and only if $\uu$ is almost $G$-rich.
\end{enumerate}

\end{thm}

\begin{proof}
The first part follows from \Cref{lem:D_G_a_G_lps} and \Cref{CharakteristikaLongestSuffix}.

To show the second part, one can see that it follows from \Cref{lem:D_G_a_G_lps} that $\uu$  satisfies Property
$\Gl{G}{N}$ for some $N$ if and only if $D_G(\uu)$ is finite.
We can then use \Cref{lem:crwAlps,lem:lpsAcrw} to get equivalence with having Property $\Gc{G}{N'}$ for some $N'$.
Finally, we use \Cref{GraphImplikujeCRet,CRetImplikujeGraph} to prove the equivalence with having Property $\Gg{G}{N'}$ which is by definition equivalent with almost $G$-richness of $\uu$.
\end{proof}

\section{$G$-richness and bilateral order}

As stated in \Cref{equiv_rich}, characterization
\ref{equiv_rich_bilateral}, words rich in classical sense can be
characterized using bilateral order of bispecial factors. 
In this section we show a generalization of this statement.

The proof of this fact for classical richness  given in \cite{BaPeSta2}  is based on the validity of point \ref{equiv_rich_nasa} of  \Cref{equiv_rich}. 
The following statement is a combination of Theorem 22 and Remark 24 from \cite{PeSta1} and
it generalizes characterization \ref{equiv_rich_nasa} of \Cref{equiv_rich} for almost rich words.

\begin{prop}\label{almostG-Rich}
Let $\uu$ be an infinite word with language closed under $G$ and $N \in \N$ be a $G$-distinguishing  number on $\uu$.
The word $\uu$ satisfies Property $G$-tls$(N)$  if and only if
\begin{equation}\label{almostG-RichVztah}
\Delta \C (n) + \# G \ \ =  \sum_{\Theta \in
G^{(2)}}\Bigl(\P_{\Theta}(n) + \P_{\Theta}(n+1)\Bigr) \qquad
\hbox{for any} \ n \geq N.\end{equation}
\end{prop}

We have no modification of the previous proposition describing Property $\Gg{G}{N}$ for $N$  which are not $G$-distinguishing;
in this case we have no simple expression for the right side of the equation since it strongly depends on the relations among the elements of the group.
In \cite{PeSta1}, we show the exact expression only for groups of order $4$.
 Therefore, in this section we concentrate on the notion almost
 $G$-richness.  Let us rephrase the previous proposition in a more handy
way.
\begin{corollary}\label{prevod} Let $\uu$ be an infinite word with language
closed under $G$ and $N \in \N$ be a $G$-distinguishing number on $\uu$.
The word  $\uu$ satisfies Property
 $G$-tls$(N)$  if and only if
\begin{enumerate}
\item  $$\Delta \C (N) +
\# G \ \ =  \sum_{\Theta \in G^{(2)}}\Bigl(\P_{\Theta}(N) +
\P_{\Theta}(N+1)\Bigr)$$

\item  and  for all $n \geq N$,  we have $$\Delta^2 \C(n)=
\sum_{\Theta \in G^{(2)} \quad} \sum\limits_{\substack{w \in
\L_n(\uu)\\ w= \Theta(w)}}\left(\#\Pext_{\Theta}(w)-1\right).$$

\end{enumerate}
\end{corollary}

\begin{proof} The task to verify equalities  $a(n) =
b(n)$ for all $n\geq N$ means  to verify $a(N)=b(N)$ and $\Delta
a(n)= \Delta b(n)$ for all $n\geq N$. Let us consider $a(n)$ to be
equal to the left side  and $b(n)$ to the right side of
\eqref{almostG-RichVztah}. It is now enough to realize that $$\Delta
b(n)= \sum_{\Theta \in G^{(2)}} \P_{\Theta}(n+2)-\P_{\Theta}(n) =
\sum_{\Theta \in G^{(2)} \quad} \sum\limits_{\substack{w \in
\L_n(\uu)\\ w= \Theta(w)}}\left(\#\Pext_{\Theta}(w)-1\right),$$
where we used equalities \eqref{DeltaP} and $\P_{\Theta}(n) = \sum\limits_{\substack{w \in
\L_n(\uu)\\ w= \Theta(w)}} 1$.
\end{proof}

\begin{prop} \label{G_rich_Pext_nerovnosti}
Let $N \in \N$, $\uu \in \A^\N$ satisfy  $\Gg{G}{N}$, and $w$ be a bispecial factor
of $\uu$ of length at least $ N$.

\begin{itemize}
    \item If $w$ is not a $G$-palindrome, then
$$\b(w) \geq 0.$$
\item If $w$ is a $\Theta$-palindrome for an antimorphism $\Theta
\in G$, then
$$\b(w)\geq \# \Pext_{\Theta}(w) - 1.$$
\end{itemize}
\end{prop}

\begin{proof}
Let $w$ be a bispecial factor having its length $M:=|w| \geq  N$
such that for all antimorphism $\Theta \in G$, $\Theta(w) \neq w$.
By the definition of $\b(w)$, we want to prove
\begin{equation}\label{zasestrom} \# \Bext(w)
\geq \# \Rext(w) + \# \Lext(w) - 1.
\end{equation}

Let $B(w)$ be a bipartite graph with the set of vertices
$$
V(w) = \left \{ aw \mid a \in \Lext(w) \right \} \cup \left \{ wb
\mid b \in \Rext(w) \right \}.
$$
There is an edge connecting  vertices $aw$ and $wb$ if the word
$awb$ is a factor of $\uu$. The number of vertices in the graph
$B(w)$ is $\# \Rext(w) + \# \Lext(w)$ and the number of edges is
$\# \Bext(w)$. In the sequel,  we show that this graph is connected. Since
in any connected graph the number of edges equals at least the
number of vertices minus one, the inequality \eqref{zasestrom}
follows.

Let $(k_n)$ be an increasing sequence of indices such that $k_0 >
0$ is an occurrence of $w$ and $k_n$ is an $G$-occurrence of $w$
for any $n\geq 1$. Moreover, any $G$-occurrence of $w$ in the
suffix $u_{k_0}u_{k_0+1}u_{k_0+2}\ldots $  of $\uu$ belongs to the
sequence $(k_n)$. As $\uu$ satisfies Property $\Gg{G}{N}$, then according
to \Cref{rem:alternuji} we have

$u_{k_1}u_{k_1+1}\cdots u_{k_1+M-1} = \nu_1(w)$, where  $\nu_1\in
G$ is an antimorphism,

$u_{k_2}u_{k_2+1}\cdots u_{k_2+M-1} = \nu_2(w)$, where  $\nu_2\in
G$ is a morphism,

$u_{k_3}u_{k_3+1}\cdots u_{k_3+M-1} = \nu_3(w)$, where  $\nu_3\in
G$ is an antimorphism,

and so on.\\
The restriction of  $\nu \in G$ to the set of letters is just a
permutation. Therefore, for any $\nu\in G$ and any $b\in
\mathcal{A}$, there exists a letter $a\in \mathcal{A}$ such that
$b=\nu(a)$. Thus,  for any $n\in \mathbb{N}$, the factor
$u_{k_n-1}u_{k_n}u_{k_n+1}\cdots u_{k_n+M}$ can be written as
$$u_{k_n-1}u_{k_n}u_{k_n+1}\cdots u_{k_n+M}=\nu_n(c_n)\nu_n(w)\nu_n(d_n)\quad \hbox{for some letters
$c_n$ and $d_n$.}$$
 As $\nu_{2i}$ is  a morphism and $\nu_{2i+1}$
is an antimorphism, we have
\begin{equation}\label{bilat}
c_{2i}wd_{2i} \in \L(\uu) \quad \hbox{ and}\quad d_{2i-1}wc_{2i-1}
\in \L(\uu).
\end{equation}
Because $w$ is not a $G$-palindrome, any $G$-occurrence of $w$
together with the
 left and the right neighboring letters corresponds  to a unique edge in the graph $B(w)$.
For any $n\in \mathbb{N}$, the factor
$$\nu_n(w)\nu_n(d_n) \cdots \nu_{n+1}(c_{n+1})\nu_{n+1}(w)$$ is a
complete $G$-return word of $[w]$. According to
\Cref{lem:rich_na_crw}, there exists an antimorphism $\Theta\in G$
such that $\nu_{n+1}(w) = \Theta \bigl(\nu_{n}(w)\bigr)$ and
$\nu_{n+1}(c_{n+1})=\Theta\bigl(\nu_n(d_{n})\bigr)$.    As
$\nu_{n+1}(w) = \Theta \bigl(\nu_{n}(w)\bigr)$ implies $\nu_{n+1}
= \Theta \nu_{n}$, we get $\nu_{n+1}(c_{n+1})=
\Theta\bigl(\nu_n(c_{n+1})\bigr) = \Theta\bigl(\nu_n(d_{n})\bigr)$
and thus $c_{n+1}=d_n$. Using \eqref{bilat} we obtain
\begin{equation*}\label{bilat2}
c_{2i}wc_{2i+1} \in \L(\uu) \quad \hbox{ and}\quad c_{2i}wc_{2i-1}
\in \L(\uu).
\end{equation*}
Recurrence of $\uu$ implies
$$ V(w) = \{ c_{2i}w \mid i\in \mathbb{N}\} \cup \{ wc_{2i-1} \mid i\in
\mathbb{N}\}.$$ Walking along $\uu$, each $G$-occurrence of $w$
represents an unordered edge connecting $c_{2i}w$ with $wc_{2i-1}$
or $c_{2i}w$ with $wc_{2i+1}$, and thus in fact walking along
$\uu$ represents a walk in the graph $B(w)$. Since any factor $aw$
and $wb$ occurs in $\L(\uu)$ infinitely many times, this walk in
the graph $B(w)$ uses all vertices of  $B(w)$. Therefore, the
graph $B(w)$ is connected.

\vspace{\baselineskip}

Now consider a  $G$-palindromic bispecial factor $w$ and  denote
by $\Theta$ the antimorphism  such that $w=\Theta(w)$. We define the
bipartite graph $B(w)$  in the same way as before. If $awb\in
\L(\uu)$ and $ b\neq \Theta(a)$, then $B(w)$ contains with the edge
$  awb$ also the different  edge $\Theta(b)w\Theta(a)$. Therefore,
any $G$-occurrence of $w$  together with the left and the right
neighboring letters corresponds  to a pair of edges in the graph
$B(w)$ unless it represents a $\Theta$-palindromic extension
$aw\Theta(a)$. Let us  replace the graph $B(w)$ by the graph $
B'(w)$ in which vertices are couples $\{aw, w\Theta(a)\}$ and
edges are either couples  $ \{awb, \Theta(b)w\Theta(a)\}$  or
loops representing a $\Theta$-palindromic extension
$\{aw\Theta(a)\}$. Now  we can interpret a walk along $\uu$ as a
walk in the new graph $B'(w)$. Consequently, the  graph $B'(w)$ must be
connected. The connectivity of $B'(w)$ implies that the number of
edges in $B'(w)$ which are not loops is at least equal to $ \#
\Rext(w) - 1$.  Since
$$  \#\{ awb \mid b\neq \Theta(a)\} = 2\times \hbox{number of  edges in $B'(w)$ which are not loops,}$$
we obtain $$\# \Bext(w)= \#\{ awb \mid b\neq \Theta(a)\} + \#
\Pext_\Theta(w) \geq 2 (\# \Rext(w) -1) + \# \Pext_\Theta(w)  $$
As $\# \Rext(w)= \# \Lext(w)$, we deduce
\begin{equation*}
b(w) = \# \Bext(w) - \# \Rext(w)-\# \Lext(w)+1 \geq \#
 \Pext_\Theta(w)- 1. \tag*{\qedhere}
\end{equation*}
\end{proof}

\begin{prop}\label{G_rich_Pext_rovnosti}
Let $\uu$ be an infinite word with language closed under $G$ and $N \in \N$ be $G$-distinguishing on $\uu$.
 The word $\uu$ has Property
$\Gg{G}{N}$  if and only if
 any bispecial factor $w$ of
$\uu$ of length at least $N$  satisfies:
\begin{itemize}
    \item if $w$ is not a $G$-palindrome, then
$$\b(w) = 0,$$
    \item if $w$ is a $\Theta$-palindrome for some $\Theta \in G$, then
$$\b(w) = \# \Pext_{\Theta}(w) - 1;$$
\end{itemize}
and
$$
\Delta \C (N) + \# G \ \ =  \sum_{\Theta \in
G^{(2)}}\Bigl(\P_{\Theta}(N) + \P_{\Theta}(N+1)\Bigr).
$$
\end{prop}

\begin{proof}
\vspace{\baselineskip}

\noindent $(\Leftarrow)$:  The assumption on bilateral orders and
the fact that non-bispecial $\Theta$-pa\-lindromic factors have
a~unique $\Theta$-palindromic extension guarantee the following equality
for all $n \geq N$:
\begin{equation}\label{sec_diff_pal}
\Delta^2 \C(n)= \sum_{w \in \L_n(\uu)}\b(w) = \sum_{\Theta \in
G^{(2)} } \ \sum_{\substack{w \in \L_n(\uu)\\ w=
\Theta(w)}}\left(\#\Pext_{\Theta}(w)-1\right).
\end{equation}
According to \Cref{prevod} it implies that $\uu$ satisfies Property $\Gg{G}{N}$.

\noindent $(\Rightarrow)$:
Let $n \geq N$.
Using
\Cref{G_rich_Pext_nerovnosti} and \Cref{prevod} we obtain

\begin{align*}
\Delta^2 \C(n) &= \sum_{w \in \L_n(\uu)}\b(w) =  \sum_{\substack{w
\in \L_n(\uu)\\ w\neq \Theta(w)\\ \hbox{\tiny for\  all}\
\Theta}}\b(w) +
\sum_{\Theta \in G^{(2)} \quad} \sum_{\substack{w \in \L_n(\uu)\\
w= \Theta(w)}}\b(w) \\
& \geq  \sum_{\Theta \in G^{(2)} \quad} \sum_{\substack{w \in \L_n(\uu)\\
w= \Theta(w)}}\left(\#\Pext_{\Theta}(w)-1\right) = \Delta^2 \C(n).
\end{align*}
As the beginning and the end of our  estimates is the same, the inequalities for $\b(w)$ deduced in
\Cref{G_rich_Pext_nerovnosti} must be equalities, which was
to prove.
\end{proof}

%%%%%%%%%%%%%%%%%%%%%%%%%%%%%%%%%%%%%%%%%%%%%%%%%%%%%

%=========================================================================================
%=========================================================================================
%=========================================================================================
\section{Examples}

The aim of this section is to exhibit examples of $G$-rich words.
As already mentioned in \Cref{ex:genTM}, it is shown in \cite{Sta2011} that for any dihedral group $I_2(m)$ there exist words, namely word $\tt_{b,m}$ for all integers $b \geq 2$, such that they are $I_2(m)$-rich.
Dihedral groups form part of finite Coxeter groups which, according to \Cref{thm:Ggenerovane}, are good candidates for a group $G$ when looking for an example of $G$-rich word.
In this section we provide $2$ examples of $G$-rich words such that $G$ is not a dihedral group.

\newsavebox{\tempbox}
\begin{figure}[h!]%
\sbox{\tempbox}{\begin{minipage}[t]{0.45\textwidth}
\centering\includegraphics[scale=1.2]{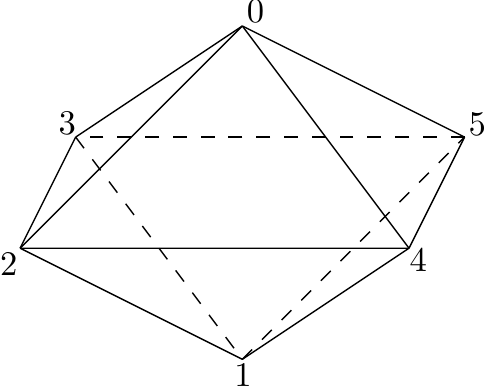}\vspace{\baselineskip}\end{minipage}}%
\centering
\subfloat[][]{ \label{fig:koso} %
\begin{minipage}[t]{0.45\textwidth}
\centering
\vbox to \ht\tempbox{%
\vfil
\includegraphics[scale=1.2]{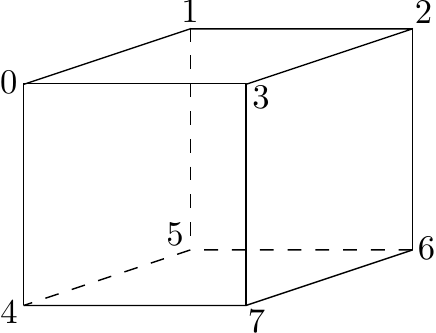}
\vfil}
\vspace{\baselineskip}
\end{minipage}}%
\subfloat[][]{ \label{fig:pyra} \usebox{\tempbox}}%
\caption{(a) Symmetries of $\Lu$ from \Cref{ex:8} represented by the symmetries of an orthogonal prism with rhomb base. (b) Symmetries of $\L(\vv)$ from \Cref{ex:6} represented by two pyramids joint together by their rectangular bases.}
\label{fig:kopy}
\end{figure}

The first group, denoted $G$, is constructed on an $8$-letter alphabet $\mathcal{A}:=\{0,1,\ldots,7\}$.
The antimorphisms  $\Theta_0, \Theta_1$ and $\Theta_2$ are defined on $\mathcal{A}^*$ as follows
\begin{eqnarray*}
\Theta_0:&0 \mapsto 2, 1 \mapsto 1, 2 \mapsto 0, 3 \mapsto 3, 4 \mapsto 6, 5 \mapsto 5, 6 \mapsto 4, 7 \mapsto 7,& \\
\Theta_1:&0 \mapsto 4, 1 \mapsto 5, 2 \mapsto 6, 3 \mapsto 7, 4 \mapsto 0, 5 \mapsto 1, 6 \mapsto 2, 7 \mapsto 3,& \\
\Theta_2:&0 \mapsto 0, 1 \mapsto 3, 2 \mapsto 2, 3 \mapsto 1, 4 \mapsto 4, 5 \mapsto 7, 6 \mapsto 6, 7 \mapsto 5.&
\end{eqnarray*}
The group  $G \subset AM(\mathcal{A}^*)$  is  the group generated by $\Theta_0, \Theta_1$ and $\Theta_2$.
If we label the vertices of an orthogonal prism with rhomb base by the letters of $\A$ as depicted in \Cref{fig:koso}, then the antimorphisms of $G$ correspond to the mirror symmetries of the prism.

The second group, denoted $H$, is on a $6$-letter alphabet $\mathcal{B}:=\{0,1,\ldots,5\}$ and
 $H \subset AM(\mathcal{B}^*)$   is generated by the $3$ following antimorphisms:
\begin{eqnarray*}
\Psi_0:&0 \mapsto 0, 1 \mapsto 1, 2 \mapsto 4, 3 \mapsto 5, 4 \mapsto 2, 5 \mapsto 3,& \\
\Psi_1:&0 \mapsto 1, 1 \mapsto 0, 2 \mapsto 2, 3 \mapsto 3, 4 \mapsto 4, 5 \mapsto 5,& \\
\Psi_2:&0 \mapsto 0, 1 \mapsto 1, 2 \mapsto 3, 3 \mapsto 2, 4 \mapsto 5, 5 \mapsto 4.&
\end{eqnarray*}
The antimorphisms generating the group $H$ can be visualised by the mirror symmetries of the object depicted in \Cref{fig:pyra}.

In fact, the groups $G$ and $H$ are isomorphic to $\Z_2 \times \Z_2 \times \Z_2$.
They may be viewed as group actions of the group $\Z_2 \times \Z_2 \times \Z_2$ on distinct free monoids: $G$ on $\A^*$ and $H$ on $\mathcal{B}^*$.
%What differs is the action of this abstract group: $G$ acts on $\A$ and $H$ acts on $\mathcal{B}$.

\begin{example} \label{ex:8}
Let $\varphi: \mathcal{A}^* \mapsto \mathcal{A}^*$ be a morphism defined as
$$
\varphi: 0 \mapsto 01, 1 \mapsto 2, 2 \mapsto 65, 3 \mapsto 4, 4 \mapsto 23, 5 \mapsto 6, 6 \mapsto 47, 7 \mapsto 0.
$$
Denote by $\uu$ the fixed point   of $\varphi$.
In \Cref{sub:ex8}, we show that $\uu$  has its language closed under $G$ and $\uu$ is $G$-rich.
\end{example}

\begin{example} \label{ex:6}
Let $\mu: \mathcal{A}^* \mapsto \mathcal{B}^*$  be a morphism defined as
$$
\mu: 0 \mapsto 15, 1 \mapsto 04, 2 \mapsto 12, 3 \mapsto 03, 4 \mapsto 04, 5 \mapsto 12, 6 \mapsto 03, 7 \mapsto 15.
$$
Let $\vv = \mu(\uu)$.
In \Cref{sec:ex_v}, we show that $\L(\vv)$ is closed under $H$ and $\vv$  is $H$-rich.
\end{example}

The proofs of properties of $\uu$ and $\vv$ are split into several lemmas.  Instead of their complete proofs, we provide just  sketches  or hints for readers.

%=========================================================================================
\subsection{$G$-richness of $\uu$} \label{sub:ex8}

To prove that $\uu$ defined in Example \ref{ex:8}  is $G$-rich, we show in the sequel that $\uu$ has Property
 $\Gg{G}{1}$.  For this reason, we   exploit Proposition \ref{G_rich_Pext_rovnosti}. Therefore, one  needs to study  bispecial factors occurring  in $\uu$.
We use a general method for circular morphisms described in \cite{Kl12}.
In the case of $\varphi$, it yields simple results.
To describe the bispecial factors of $\uu$, we introduce the mapping   $\pi:\{0,2,4,6\} \mapsto  \{0,2,4,6\}$ as follows
$$
\pi: 0 \mapsto 2, 4 \mapsto 0, 2 \mapsto 4, 6 \mapsto 6.
$$
\begin{lem} \label{lem:ex8_BS_dolu}
Let $w = w_0 \cdots w_{n-1}$ be a nonempty bispecial factor of $\uu$. Then $w_{n-1} \in \{0,2,4,6\}$ and
 $\varphi(w) \pi(w_{n-1})$
is a bispecial factor of $\uu$.
Moreover, $\b(w) = \b \left( \varphi(w) \pi(w_{n-1}) \right)$.
\end{lem}

\begin{proof}
The claim follows from the definition of $\varphi$ and the fact that right special factors of length $1$ are factors $0, 2, 4$ and $6$,
and each has $2$ right extensions.
\end{proof}

The next statement can be easily deduced from the form of $\varphi$ as well.

\begin{lem}  \label{lem:ex8_BS_nahoru}
Let $w$, $|w| \geq 2$, be a bispecial factor of $\uu$.
Then there exists a unique bispecial factor of $\uu$, say $v = v_0v_1\cdots v_{m-1}$, such that
\begin{equation}\label{eq:novyPalindrom}
w = \varphi(v) \pi(v_{m-1}).
\end{equation}
\end{lem}

According to the last two lemmas, all bispecial factors can be  constructed from  bispecial factors of length $1$ using recursively formula \eqref{eq:novyPalindrom}. In fact, all the letters of $\mathcal{A}$ are bispecial factors  and $G$-palindromes.    We show that the formula \eqref{eq:novyPalindrom}   produces  from a  $G$-palindrome again   a   $G$-palindrome.

\begin{lem} \label{lem:ex8_kom}
For all $i \in \Z_3$ and $w \in \Lu$, $w = w_0 \cdots w_{n-1}$, we have
$$
x_{i-1}(w)\Theta_i\varphi(w) = \varphi\Theta_{i-1}(w)y_{i}(w),
$$
where $y_i(w) := \Theta_i \big( \varphi(w_0) \big)_0$ and $x_i(w) = \big( \varphi \Theta_i (w_{n-1}) \big) _ 0$ and
$\big( v \big)_0$ denotes the first letter of a word $v$.
\end{lem}

\begin{proof}[Sketch of the proof]
The proof is done by induction on $n$.
Supposing the claim holds for $n$, one needs to deal with different cases according to the value $i$ and possible factors $w_{n-1}w_n \in \L_2(\uu) = \{54, 62, 47, 12, 04, 76, 65, 40, 01, 23, 30, 26\}$.
The claim then follows from the definitions of $\varphi$, $\Theta_i$ and $\Theta_{i-1}$.
\end{proof}

\begin{lem}\label{pasuji}
Let $w = w_0 \cdots w_{n-1}$ be a nonempty bispecial factor of $\uu$.
Then $w$ is a \mbox{$\Theta$-palindrome}, $\Theta \in G$, and $\b(w) = \# \PextT(w) - 1$.
\end{lem}

\begin{proof}
As the bilateral order of a bispecial factor of length $1$ equals $0$,
according to \Cref{lem:ex8_BS_dolu,lem:ex8_BS_nahoru},
all bispecial factors have their bilateral order equal to $0$.
It is also clear that they have $2$ right and $2$ left extensions.

If $w = w_0 \cdots w_{n-1}$ is a nonempty bispecial factor and if $w$ is a $\Theta_{i-1}$-palindrome
for $i \in \Z_3$,
one can show that $\pi(w_{n-1}) =  y_{i}(w)$ and
thus, $\varphi(w) \pi(w_{n-1})$ is a $\Theta_i$-palindrome.
Since the bispecial factors of length $1$ are $\Theta_2$-palindromes,
all bispecial factors are $\Theta_i$-palindromes for some $i \in \Z_3$.

It follows that for all $i \in \Z_3$, $\Lu$ contains infinitely many $\Theta_i$-palindromes.
Therefore, $\Lu$ is closed under $G$.

Let $w$ be a nonempty bispecial factor.
Since $\b(w) = 0$, $w$ has $2$ left and $2$ right extensions, $w$ is a $\Theta$-palindrome for a unique $\Theta \in G$,
and $\Lu$ is closed under $G$,
one can see that the number of $\Theta$-palindromic extensions of $w$ is $1$.
\end{proof}

\begin{proof} [Proof of $G$-richness of $\uu$ defined  in  \Cref{ex:8}]

 At first, we realize that
 the  generators  $\Theta_0, \Theta_1$ and $\Theta_2$  of the group $G$ guarantee the  number $1$ to be  $G$-distinguishing on any infinite word over $\A$.
Because of \Cref{pasuji} and  \Cref{G_rich_Pext_rovnosti},
it remains to verify that  $\Delta \C(1) +\# G $ equals the number of all $G$-palindromes of length $1$ and $2$.
One can readily  see that  $\Delta \C(1) = 4$, $\# G = 8$,  the number of $G$-palindromes of length $1$ is $8$ and the number of $G$-palindromes of length $2$ is $4$.
\end{proof}

%=========================================================================================
\subsection{$H$-richness of $\vv$} \label{sec:ex_v}

The proof of $H$-richness of $\vv$ is very similar to the previous proof.
In order to use \Cref{G_rich_Pext_rovnosti},
we explore the bilateral orders and $H$-palindromic extensions of bispecial factors of $\vv$.

We define the morphism $\eta:  \A^* \mapsto \B^*$ as
$$ \eta: 0 \mapsto 041, 1\mapsto 120, 2 \mapsto 031, 3\mapsto 150, 4 \mapsto 150, 5 \mapsto 041, 6\mapsto 120, 7 \mapsto 031.$$
Let $w = w_0 \cdots w_{n-1} \in \A^*$ be a nonempty factor of $\uu$.
It follows from $\L_2(\uu)$ and the definition of $\mu$ that $\mu(w) \eta( w_{n-1})$ is a factor of $\vv$.
The following lemma summarizes the relation between the bispecial factors of $\vv$ and of $\uu$.

\begin{lem} \label{lem:ex6_BS}
Let $w \in \Lu$, $w = w_0 \cdots w_{n-1}$, be a nonempty bispecial factor.
Then $\mu(w) \eta( w_{n-1})$ is a bispecial factor of $\vv$.

On the other hand, if $v \in \L(\vv)$, $|v| \geq 5$, is a bispecial factor of $\vv$, then there exists a unique nonempty bispecial factor $w \in \Lu$ such that $v = \mu(w) \eta( w_{n-1})$.
\end{lem}

\begin{lem} \label{lem:ex6_kom}
Let $i \in \Z_3$.
If $w \in \Lu$ is a nonempty $\Theta_i$-palindrome,
then the factor $\mu(w) \eta(w_{n-1}) \in \L(\vv)$ is a $\Psi_i$-palindrome.
\end{lem}

\begin{proof}[Sketch of the proof]
We induce on the length of $w$.
Fix $i \in \Z_3$.
Suppose the claim holds for $w = w_0 \cdots w_{n-1} = \Theta_i(w)$.
Take $z \in \A$ such that $zw\Theta_i(z) \in \Lu$.
The proof follows from the definition of $\mu$, $\eta$, and possible factors $zw_0 \in \L_2(\uu)$.
\end{proof}

\begin{proof}[Proof of $H$-richness of $\vv$ defined  in  \Cref{ex:6}]
According to the previous lemma, it is clear that $\L(\vv)$ is closed under $H$.
The properties of $\Lu$ also imply that all bispecial factors of $\vv$ of length greater than or equal to $5$ have bilateral order $0$ and one $\Theta$-palindromic extension, where $\Theta \in H$ is the unique antimorphism fixing the bispecial factor.
For shorter bispecial factors, of length greater than $1$, this property needs to be verified by hand and is left to the reader.

Since $2$ is an $H$-distinguishing number on $\vv$, \Cref{G_rich_Pext_rovnosti} requires to evaluate $\Delta \C(2)$, $\PT(2)$ and $\PT(3)$ for all involutive antimorphism $\Theta \in H$.
It is easy to verify that $\Delta \C(2) = 4$, $\sum_{\Theta \in H^{(2)}} \PT(2) = 0$ and $\sum_{\Theta \in H^{(2)}} \PT(3) = 12$.
Since $\# H = 8$, according to \Cref{G_rich_Pext_rovnosti}, $\vv$ satisfies Property $\Gg{H}{2}$.

To claim that $\vv$ is $H$-rich, we need to verify that $\vv$  satisfies $\Gg{H}{1}$.
Thus, it remains to show that all loops in $\Gamma_1(\vv)$ are $H$-palindromes and the graph obtained from $\Gamma_1(\vv)$ by removing loops is a tree.
Since it can be easily verified by hand, the word $\vv$ is $H$-rich.
\end{proof}

Denote for all $i \in \Z_3$ by $H_i$ the subgroup of $H$ generated by $\Psi_i$ and $\Psi_{(i+1 \mod 3)}$.
It is easy to verify that $\# H_i = 4$ for all $i$, the number $1$ is $H_0$-distinguishing and $H_1$-distinguishing, and the number $2$ is $H_2$-distinguishing.
It follows from the last proof that $\P_{\Psi_i}(3) = 4$ and $\P_{\Psi_i}(2) = 0$ for all $i \in \Z_3$.
One can also verify that $\P_{\Psi_0}(1) = \P_{\Psi_2}(1) = 2$ and $\P_{\Psi_1}(1) = 4$.
Since $\Delta \C(1) = 2$, using \Cref{G_rich_Pext_rovnosti} we get that the word $\vv$ is $H_0$-rich, $H_1$-rich.
Since $\Delta \C(2) = 4$, again using \Cref{G_rich_Pext_rovnosti} we get that the word $\vv$ is almost $H_2$-rich (it satisfies $\Gg{H_2}{2}$). In fact, it can be shown that the word $\vv$ satisfies $\Gg{H_2}{1}$ and thus it is also $H_2$-rich.

\section{Comments and open problems}

\begin{itemize}

\item  The dihedral groups $I_2(m)$  form a special class of finite
Coxeter groups which belong to a broader class of groups generated by involutive elements.
As shown in \cite{Sta2011} and recalled in \Cref{ex:genTM}, for any dihedral group there exists a
$I_2(m)$-rich word.
Is it possible for any given finite group generated by involutive antimorphisms or at least a given finite Coxeter group $G$ to find a $G$-rich word?

We believe that an approach using a generalized palindromic closure operator as introduced in the last chapter of \cite{LuLu} might be helpful.

\item For $\# G > 2$, the list of examples of $G$-rich words is very short and the
list of almost $G$-rich words (which are not $G$-rich) is empty.
In \cite{GlJuWiZa}, Glen et al. described  a class of morphisms such that
morphic image of a rich word under a morphism from this class  has a finite
 nonzero defect. Find a class of morphisms producing almost
$G$-rich words with finite nonzero defect by applying a
morphism from this class to a $G$-rich word.

\item For a word $\uu$ with language closed under reversal, Brlek
and Reutenauer conjectured  in  \cite{BrRe-conjecture}  for
the defect $D(\uu)$ that
$$ 2D(\uu) = \sum_{n\in \mathbb{N}} T(n), \quad
\hbox{where}\ \  T(n) := \Delta \C (n) + 2 - \P(n+1) - \P(n).$$
The conjecture was shown in \cite{BaPeSta5}.

Can the $G$-defect of an infinite word $\uu$ be expressed  using
the differences between right and left sides of inequalities  in
\eqref{mistoStromovi}?

\item In \Cref{sec:clasika}, definitions of rich words and
 $\Theta$-rich words were reminded. In our new terminology,
they are  $\{\id,R\}$-rich words and  $\{\id,\Theta\}$-rich words
respectively. The groups $\{\id, R\}$ and $\{\id,\Theta\}$ are
clearly isomorphic.
In \cite{BuLu}, it is shown that a so-called $\Theta$-standard word with seed,
which is almost $\Theta$-rich, is a morphic image of a standard Arnoux-Rauzy words, which is rich.
In \cite{PeSta_Milano_IJFCS},  we have a more general case:
any uniformly recurrent almost $\Theta$-rich word is a
morphic image of a rich word. Is an almost $G_1$-rich word
related to a $G_2$-rich word for some group $G_2$ isomorphic to
$G_1$?

\item Let $\uu$ be an infinite word having language closed under a group $G$.
The closedness under $G$ can be exploited to estimate the number of distinct frequencies of factors of the same length $n$.
In \cite{Ba2011}, an upper bound on this number is given (for $n$ being a $G$-distinguishing number).
The estimate is based on the inequality from \Cref{nerovnostProVice}.
Looking at the proof of the estimate, it can be seen that the only candidates for reaching the upper bound for all sufficiently large $n$ are almost $G$-rich words. However, as noted in \cite{Ba2011}, in our words, almost $G$-richness does not imply the upper bound to hold for all sufficiently large $n$.

\item The Thue-Morse word is $G$-rich, where $G$ is generated by two commuting antimorphisms $R$ and $E$.
However, it is not $G_1$-rich while taking a proper subgroup $G_1$ of $G$.
In our considerations, we did not assume the group $G$ to be the maximal group of symmetries such that an infinite word $\uu$ is closed under $G$.

Suppose $\uu$ is an infinite word having language closed under a group $G$.
Let $G_1$ be a proper subgroup of $G$ containing at least one antimorphism.
Suppose $\uu$ is both almost $G$-rich and almost $G_1$-rich.
Let $N$ be a $G$-distinguishing number.
Then, according to \Cref{almostG-Rich}, we can for all $n \geq N$ write
\begin{eqnarray*}
\Delta \C (n) + \# G &=&  \sum_{\Theta \in G^{(2)}} \Bigl(\P_{\Theta}(n) + \P_{\Theta}(n+1)\Bigr) \text \quad { and } \\
\Delta \C (n) + \# G_1  &=&  \sum_{\Theta \in G_1^{(2)}} \Bigl(\P_{\Theta}(n) + \P_{\Theta}(n+1)\Bigr).
\end{eqnarray*}
Thus, we get
\begin{equation}\label{com:odecetene}
\#G - \#G_1 = \sum_{\Theta \in G^{(2)} \setminus G_1^{(2)}} \Bigl(\P_{\Theta}(n) + \P_{\Theta}(n+1)\Bigr).
\end{equation}
Since $G_1$ is a proper subgroup of $G$, we have $\#G = \ell \# G_1$ for $\ell > 1$.
Take $\Theta \in G^{(2)} \setminus G_1^{(2)}$ and $w \in \Lu$, $|w| \geq N$, such that $\Theta(w) = w$.
Then for all $v \in [w]$, it can be show that there exists $\Psi \in G^{(2)} \setminus G_1^{(2)}$ such that $v$ is a $\Psi$-palindrome.
Since $\# [w] = \frac{\#G}{2}$, the right side of \eqref{com:odecetene} equals $k(n) \frac{\#G}{2}$ for $k(n) \in \N$.
We get from \eqref{com:odecetene} that
$$
( \ell - 1) \# G_1 = k(n) \ell \frac{\#G_1}{2}.
$$
The only solution is $\ell = 2$ and $k(n) = 1$ for all $n$.
Thus, we obtain the following condition
$$
\# G_1 = \frac{1}{2} \# G = \sum_{\Theta \in G^{(2)} \setminus G_1^{(2)}} \Bigl(\P_{\Theta}(n) + \P_{\Theta}(n+1)\Bigr)
\quad \text{ for all $n \geq N$.}
$$
Indeed, these conditions are satisfied for the three subgroups $H_0$, $H_1$ and $H_2$ of the group $H$ and the word $\vv$, see the last part of \Cref{sec:ex_v}.

Further characterization of such group and examples of such infinite words is an open problem.

\end{itemize}

\section*{Acknowledgments}

We would like to express our gratitude to the anonymous referee of this articles.
His or her review helped us to improve the presentation and also remove some flaws concerning Coxeter groups.
This work was supported by the Czech Science Foundation grants GA\v
CR 201/09/0584, 13-03538S, 13-35273P
and by the grant of the Grant Agency of the Czech Technical
University in Prague grant No. SGS11/162/OHK4/3T/14.

\end{document}